\documentclass[12pt,a4paper,reqno]{amsart}
\usepackage{amsmath,amsfonts,amssymb,amsthm,xcolor,amscd,tikz,xspace,enumerate}
\usepackage[mathscr]{eucal}
\usepackage[pagebackref, colorlinks=true,linkcolor=blue,citecolor=red]{hyperref}
\usepackage{pdfpages}
\usetikzlibrary{arrows}
\usetikzlibrary{intersections}
\usepackage{pgfplots}
\usetikzlibrary{calc,3d,shapes, pgfplots.external, intersections}
\usepackage{pgfplots}
\usepackage{braket}
\usepackage{float}
\usetikzlibrary{patterns}
\usetikzlibrary{shadings,intersections,calc}

\newtheorem{thm}{Theorem}[section]
\newtheorem{cor}[thm]{Corollary}
\newtheorem{lem}[thm]{Lemma}
\newtheorem{conj}[thm]{Conjecture}
\newtheorem{prop}[thm]{Proposition}
\theoremstyle{definition}
\newtheorem{defn}[thm]{Definition}
\theoremstyle{remark}
\newtheorem{rem}[thm]{Remark}
\newtheorem{ex}[thm]{Example}
\numberwithin{equation}{section}

\begin{document}

\title[On the  Hamilton's isoperimetric ratio in...]
{On the Hamilton's isoperimetric ratio in complete  Riemannian manifolds of finite volume}

\author[S. Nardulli]{Stefano Nardulli}
\address{
Centro de Matem\'atica Computa\c{c}\~{a}o e Cogni\c{c}\~{a}o\endgraf
Universidade Federal do ABC \endgraf 
Santo Andr\'e, SP, Brazil \endgraf  
and\endgraf
Department of Mathematics\endgraf
Princeton University \endgraf
Fine Hall, Washington Road, Princeton, NJ 08544-1000, USA
}
\email{stefano.nardulli@ufabc.edu.br, stefanon@math.princeton.edu}

\author[F.G. Russo]{Francesco G. Russo}
\address{Department of Mathematics and Applied Mathematics\endgraf
University of Cape Town \endgraf
Private Bag X1, 7701, Rondebosch\endgraf
Cape Town, South Africa\endgraf
and\endgraf
Department of Mathematics and Applied Mathematics\endgraf
University of the Western Cape\endgraf 
Bellville, 7535, South Africa}
\email{francescog.russo@yahoo.com}

\date{\today}

%% Four or five keywords or phrases

\begin{abstract} 
We study a family of geometric variational functionals introduced by Hamilton, and considered later by Daskalopulos, Sesum, Del Pino and Hsu, in order to understand the behaviour of maximal solutions of the Ricci flow both in compact and  noncompact complete Riemannian manifolds of finite volume. The  case of dimension two has some peculiarities, which force us to use different ideas from the corresponding higher-dimensional case. Under some natural restrictions, we investigate sufficient and necessary conditions which allow us to show the existence of connected regions with a connected complementary set (the so-called ``separating regions''). In dimension higher than two, the associated problem of minimization is reduced to an auxiliary problem for the isoperimetric profile (with the corresponding investigation of the minimizers). This is possible via an argument of compactness in geometric measure theory valid for the case of complete finite volume manifolds. Moreover, we show that the minimum of the separating variational problem is achieved by an isoperimetric region. The dimension two requires different techniques of proof. The present results develop a definitive theory, which allows us to circumvent the shortening curve flow approach of the above mentioned authors at the cost of some applications of the geometric measure theory and of the Ascoli-Arzela’s Theorem.

%We study a family of geometric variational functionals introduced by Hamilton, and considered also in works of Daskalopulos, Sesum, Del Pino and Hsu, in order to understand the behaviour of maximal solutions of the Ricci flow before in compact, and later in complete noncompact  orientable Riemannian manifolds of finite volume of dimension two. The case of dimension two has some peculiarities, which force us to use different ideas from the corresponding higher-dimensional case. Under necessary and sufficient restrictions imposed on the functionals introduced by the authors above, we show the existence of connected regions with a connected complementary set (the so-called "separating regions''). In dimension higher than two, the associated problem of minimization is reduced to an auxiliary problem for the isoperimetric profile (with the corresponding investigation of the minimizers), this is possible via an argument of compactness in geometric measure theory.  Moreover, we show that the minimizers of the separating variational problem is achieved by an isoperimetric region. In dimension two, the proof is a little bit more involved than in dimension three or higher. With the results contained in this paper we start and finish a theory, discovering a way to circumvent the shortening curve flow approach of the above mentioned authors at the cost of some cleaver applications of geometric measure theory and the Ascoli-Arzela's Theorem. 
\end{abstract}

\maketitle

\section{Introduction}

The papers \cite{hamilton3, hamilton4} have historically influenced the study of the Ricci flow on smooth Riemannian manifolds in the last 25 years. Recent advances can be found in \cite{hamilton1, hamilton2}, where Daskalopoulos and Hamilton investigate the behaviour of the maximal solutions of the Ricci flow over planes of finite volumes. Recent contributions in the same direction can also be found in \cite{hsu,  zheng}.
The original idea of Daskalopoulos and Hamilton was to introduce a series of isoperimetric ratios, which present some properties of monotonicity. These allow  to avoid singularities, which may appear at the extinction time of the Ricci flow. Again in \cite{hamilton1, hamilton2}, the authors assume the existence of minimizers for certain isoperimetric ratios, which correspond to the maximal solution of the 2--dimensional Ricci flow on a plane of finite volume. Our results deal with a proof of the existence of such minimizers in any dimension (possibly, higher than 2) under two sharp quantitative assumptions, which involve the isoperimetric profile function, allowing us to generalize \cite[Theorem 1.1]{hsu} in our Theorem \ref{ste5bis11Hsu} (note that for the reader convenience \cite[Theorem 1.1]{hsu} is reported integrally below as Theorem \ref{hsu1}). The main contributions of this work are:
\begin{enumerate}
\item In dimension three and higher, the isoperimetric problem with the separability constraint is equivalent to the one without the separability constraint.
\item In dimension two and higher, we discuss necessary and sufficient conditions, in order to show the existence of nontrivial minimizers. (See Theorems \ref{ste4}, \ref{ste61}, \ref{ste5bis}, \ref{ste4bis1}, \ref{ste5bis11}, \ref{ste5bis11Hsu} and Remark \ref{Rem:CounterExSte4}).
\end{enumerate}
The first general idea of our paper, treating the case of a Riemannian manifold $M$ of dimension $n+1\ge3$, is that the isoperimetric problem with the separability constraint is equivalent to the isoperimetric problem without the separability constraint. This equivalence holds only in dimension higher than $2$. Proposition \ref{Prop:EquivalenceInHigherDimensions0} shows the details. In dimension $2$ the equivalence fails to be true (compare Remark \ref{Rem:ComparisonWithHsu}) and different tools must be used, in order to generalize the results in literature. % \cite[Theorem 1.1]{hsu}.  

We have Theorems \ref{ste2}, \ref{ste3} \ref{ste4}, \ref{ste61}, \ref{ste5bis}, \ref{ste4bis1}, involving completely different techniques, which are proper of the geometric measure theory. Indeed, their proofs rely on an argument of compactness for finite perimeter sets in noncompact Riemannian manifolds of finite volume whose consequence is the continuity of the isoperimetric profile function. Corollaries \ref{compattezza} and \ref{continuity} are among  the new contributions that we offer on the topic of continuity and compactenss in the present context of study. To understand why  to prove continuity of the isoperimetric profile is an interesting result for itself, the reader can see also \cite{nardulli-pansu, papasoglu} in which it is shown by sophisticated examples that there exist complete Riemannian manifolds with discontinuous isoperimetric profile. Compactness for isoperimetric regions and continuity of the isoperimetric profile combined with the superadditivity property of the isoperimetric ratios (compare Lemmas \ref{algebraic},  \ref{connection}, \ref{algebraicchow}, \ref{algebraicchow1-} and  \ref{algebraicchow1}) provide the proofs of Theorems \ref{ste2}, \ref{ste3} \ref{ste4}, \ref{ste61}, \ref{ste5bis}, \ref{ste4bis1}.   
Similar arguments of compactness can be found in \cite{ritore3, morgan4, morgan1, morgan3, nardulli1, nardulli2, ritore1,  ritore2}; these contributions contain several theorems about compactness and regularity for the classical isoperimetric problem and turn out to be very powerful tools, once applied to the context of \cite{hamilton1, hamilton2}. %Their application is possible to the functionals introduced by Daskalopulos and Hamilton, because we observe a series of interesting properties of the isoperimetric functionals along with a phenomenon of superadditivity  (see Lemmas \ref{connection},  \ref{algebraicchow1}  and Corollary \ref{connectionchow}  below).

The study of the variational problems associated to the functionals of Daskalopulos and Hamilton (see Definition \ref{hamilton}) is more difficult in dimension $2$ than in dimension $3$ or higher when separability constraints are involved and we described the details in Theorems \ref{ste61}, \ref{ste5bis},  \ref{ste4bis1}, \ref{ste5bis11}, \ref{ste5bis11Hsu}. We use in dimension $2$ a soft regularizing theorem; roughly speaking, we show that ``the limit of simple curves is simple'' in the variational problem that we consider.  We do it by showing that (under our assumptions) a minimizing sequence of separating simple curves lies inside a compact set; we show that we loose perimeter in the limit, if and only if, there is more than one connected component. Again the superadditivity of the isoperimetric ratios profiles play a crucial role in the arguments of the proofs.

Our approach is completely different from the one used in the proof of  \cite[Theorem 1.1]{hsu}, and should push the theory, of the Ricci flow in dimension $2$, to wider generalizations than the original framework of  Daskalopous, Hamilton, Sesum and Del Pino \cite{delpino1, delpino2, hamilton1, hamilton2, hamilton3, hamilton4}. Indeed Theorem \ref{ste5bis11} provides new proofs and new arguments even in the compact case, generalizing \cite{hamilton3} to the noncompact case with  different techniques via curve shortening flow. To conclude this part of the introduction we highlight that our approach permits to distillate the necessary and sufficient conditions to guarantee the existence of nontrivial  minimizers. This is among our main contributions.  %Hence Theorem 1.1 is more general than the minimizer result used by Daskalopoulos and Hamilton [2].

 Section 2 is devoted to illustrate some preliminaries, which are fundamental for the proofs of the main theorems of Sections 3 and 4. We offer a new proof of the continuity of the isoperimetric profile function  (see Section 2), by means of an argument contained in \cite{ritore1}. This result has independent interest and has an important role in the structure of our proofs in Section 4. The main results are in fact here and we solve a problem of minimization for the isoperimetric ratio in the sense of Hamilton (see \cite{hamilton1, hamilton2}). In dimension two we use a classical Ascoli--Arzela Theorem to get uniform convergence after reducing the problem to the compact case. This allows us to deduce the existence of simple continuous connected separating curves in the limit, that furthermore is the boundary of an isoperimetric region see Theorems \ref{ste5bis11}, \ref{ste5bis11Hsu}. 
  Section 4 contains the proofs of the  main results of the present paper. Finally, Section 5 contains examples in which  the assumptions of the main theorems are satisfied. These examples show the usefulness of replacing the original problem with our formulation.

\section{Preliminaries}

We introduce some terminology and notation which will be used in the rest of the paper. The symbol $M^{n+1}$ denotes an open connected set of a smooth complete ($n$+1)--dimensional Riemannian manifold.  In the rest of the paper, we will write briefly $M$, in order to denote $M^{n+1}$. For any measurable set $\Omega \subseteq M$ and any open set $U \subseteq M$ (here $n \ge 0$), $\mathrm{vol}(\Omega)$ is the ($n$+1)--dimensional Hausdorff measure of $\Omega$, $\mathcal{H}^k(\Omega)$ is the $k$-dimensional Hausdorff measure of $\Omega$ (here $k \ge 0$), and 
\[\mathcal{P}(\Omega, U)=\sup\left\{\int_\Omega \mathrm{div} \ Y \ \ d\mathcal{H}^{n+1} \ | \ \|Y\|_{\infty}=1 \right\}\]
is the \textit{perimeter of} $\Omega$ \textit{relative to} $U$, where $Y$ is a smooth vector field with compact support contained in $U$, $\mathrm{div} \ Y$ denotes the divergence of $Y$, and $\|Y\|_\infty$ is the supremum norm of $Y$. Briefly, we write $\mathcal{P}(\Omega)=\mathcal{P}(\Omega, M)$ and say that $\Omega$ has \textit{finite perimeter in } $U$ if $\mathrm{vol}(\Omega \cap U) < \infty$ and $\mathcal{P}(\Omega, U)<\infty$. As  well known, these are fundamental notions in geometric measure theory, introduced by Caccioppoli \cite{caccioppoli} (via the geometric perimeter), De Giorgi \cite{degiorgi} (via the heat semigroup), and recently adapted to the context of Riemannian manifolds in \cite{miranda}. We recall from \cite{afp} that for a finite perimeter set $\Omega \subseteq M$ and an open set $U  \subseteq M$,  the \textit{reduced boundary} $\partial^* \Omega$ is the boundary of $\Omega$ in the sense of \cite[Definition 3.54, P.154]{afp}  and De Giorgi  \cite[Theorem 3.59]{afp} shows that $\mathcal{P}(\Omega,U)=\mathcal{H}^n((\partial^* \Omega) \cap U )$. 
In particular, if $\partial \Omega$ is smooth and $U=M$, then $\partial^* \Omega=\partial \Omega$ and $\mathcal{P}(\Omega)=\mathcal{H}^n(\partial \Omega)$. This point is important for the notions which we introduce in Definition \ref{hamilton}. Since we use extensively the theory of finite perimeter sets, a little technical discussion is in order here. By classical results of geometric measure theory (see Proposition $12.19$ and Formula $(15.3)$ of \cite{maggi}) we know that if $E$ is a set of locally finite perimeter in $M$ and  $B(x,1)$ an open ball of $\mathbb{R}^{n+1}$ of center $x$, radius one and $\mathrm{vol}(B(x,1))=\omega_{n+1}$, then the support of the distributional gradient measure of the characteristic function $\chi_E$ is given by $\mathrm{supp}(\nabla\chi_E)=\{x\in M \ : \ \:0< \mathrm{vol}(E\cap B(x,r))<\mathrm{vol}(B(x,r)),\forall r>0\}\subseteq\partial E$.  Furthermore there exists an equivalent Borel set $F$  such that $\mathrm{supp}(\nabla\chi_F)=\partial F=\overline{\partial^*F}$, where $\partial^*F$ is the reduced boundary of $F$. It is not too hard to show that if $E$ has $C^1$-boundary, then $\partial^*E=\partial E$, where $\partial E$ is the topological boundary of $E$. De Giorgi's Structure Theorem  \cite[Theorem 15.9]{maggi} guarantees that for every set $E$ of locally finite perimeter, the Hausdorff measure (wrt to a given metric on $M$) satisfies the condition $\mathcal{H}^n(\partial^*E)=\mathcal{P}(E)$. Therefore we may consider all  locally finite perimeter sets  (in the present paper) satisfying  $\overline{\partial^*E}=\partial E$. %It is worth to mention that the results in the book \cite{maggi} are stated and proved in $\mathbb{R}^n$ but they are valid mutatis mutandis also in an arbitrary complete Riemannian manifold, the required details could be easily provided using the work about $BV$-functions on a Riemannian manifold accomplished in \cite{miranda}.

Let's recall the notion of Hausdorff distance for metric spaces from \cite[\S VI.4]{chavel}. Given $A$ and $B$  nonempty subsets of a metric space $(X, d)$ with metric $d$, the Hausdorff distance $d_{\mathrm{H}}(A, B)$ is defined by
$$d_{\mathrm{H} }(A,B)=\max \left\{\,\sup _{a \in A}\inf _{b \in B}d(a,b),\, \sup _{b\in B}\inf _{a \in A} d(a,b)\,\right\},$$ 
and this is equivalent to consider
$$d_{\mathrm{H}}(A,B)=\inf\{\varepsilon \geq 0 \ | \ A \subseteq B_{\varepsilon }{\text{ and }}B\subseteq A_{\varepsilon }\},$$
where
$$A_{\varepsilon }=\bigcup_{a \in A}\{x\in X \ | \ d(x,a) \leq \varepsilon \},$$
is the set of all points within  $\varepsilon$ of the set $A$. Of course, $$ d_{\mathrm{H}}(A,B)=\sup_{x\in X}\left|\inf _{a\in A}d(x,a)-\inf_{b\in B}d(x,b)\right|=\sup_{x\in A\cup B}\left|\inf_{a\in A}d(x,a)-\inf_{b\in B}d(x,b)\right|,$$  where $d(x,A)$ denotes the distance from the point $x$  to the set $A$, that is,  $$d(x, A)=\inf\{d(x,a) \ : \ a \in A\}.$$

Another crucial notion is the \textit{local Hausdorff convergence}, which can be found in \cite{afp}. In other words,  if $F, F_{1}, F_{2}, \ldots $ are subsets of a Riemannian manifold $M$ with Riemannian metric $d_g$,  we say that $F_{j}$ converges locally to $F$  in the \textit{Hausdorff distance}  in $M$, if for each compact $K \subset M$  the distance  $d_\mathrm{H}(F_j \cap K, F \cap K)$  gets to zero for $j$ running to infinity. We refer to \cite{afp, chavel, morgan2, petersen} for  classical aspects of geometric measure theory and differential geometry. One of these is, for instance, the following notion. \textit{The isoperimetric profile} of $M$ is the function \[I_M  : V \in \ ]0, \mathrm{vol}(M)[ \ \longmapsto I_M(V) \in [0,\infty[\] defined by
\[I_M(V)=\inf\left\{\mathcal{P}(\Omega) \ | \ \Omega \subseteq M, \  \mathrm{vol}(\Omega)=V \right\},\] where the infimum is taken over all relatively compact open set $\Omega$ with smooth boundary.

It is good to mention here another  positive quantity, which modifies $I_M(V)$. Looking at \cite[Definition 5.79]{chow}, we recall that a smooth embedded closed (possibly disconnected) hypersurface $N \subset M$ \textit{separates} $M$, if $M-N$ has two connected components $M_1$ and $M_2$ such that $\partial M_1=\partial M_2=N$. With this notion in mind we define \[\tilde{I}_M(V)=\inf\left\{\mathcal{P}(\Omega) \ | \ \Omega \subseteq M, \  \mathrm{vol}(\Omega)=V, \  \partial \Omega \ \mathrm{is} \ \mathrm{smooth}, \ \partial \Omega \ \mathrm{separates} \ M \right\}.\] From \cite{afp, chavel, morgan2, petersen}, an \textit{isoperimetric region} in $M$ of volume $V \in \ ] 0, \mathrm{vol}(M)[$ is a set $\Omega \subseteq M$ such that $\mathrm{vol}(\Omega)=V$ and $\mathcal{P}(\Omega)=I_M(V)$. A \textit{minimizing sequence} of sets of volume $V$ is a sequence of sets of finite perimeter $\{\Omega_k\}_{k \in \mathbb{N}}$ such that $\mathrm{vol}(\Omega_k)=V$ for all $k \in \mathbb{N}$ and $\lim_{k \rightarrow \infty} \mathcal{P}(\Omega_k)=I_M(V)$.

The behaviour of a minimizing sequence for fixed volume was investigated in various contributions in the last years, but we concentrate on \cite{ritore3, morgan1, morgan3, nardulli1, nardulli2, ritore1, ritore2}, since we focus on a perspective of Riemannian geometry. The following result of Ritor\'e and Rosales \cite{ritore1} characterizes the existence of regions minimizing perimeter under a fixed volume constraint. The arguments overlap some techniques in \cite{morgan1, morgan3}.
 
 \begin{thm}[See \cite{ritore1}, Theorem 2.1]\label{rr} Let $M$ be a connected unbounded open set of a complete Riemannian manifold. For any minimizying sequence $\{\Omega_k\}_{k \in \mathbb{N}}$ of sets of volume $V$, there exist a finite perimeter set $\Omega \subset M$ and sequences of  finite perimeter sets $\{\Omega^c_k\}_{k \in \mathbb{N}}$ and $\{\Omega^d_k\}_{k \in \mathbb{N}}$ such that 
 \begin{itemize}
 \item[(i)]$\mathrm{vol}(\Omega) \le V$ and $\mathcal{P}(\Omega) \le I_M(V)$;
 \item[(ii)]$\mathrm{vol}(\Omega^c_k) + \mathrm{vol}(\Omega^d_k) =V$ and \[\lim_{k \rightarrow \infty}[\mathcal{P}(\Omega^c_k) + \mathcal{P}(\Omega^d_k)]=I_M(V);\]
  \item[(iii)]The sequence $\{\Omega^d_k\}_{k \in \mathbb{N}}$ diverges;
   \item[(iv)]Passing to a subsequence, we have that \[\lim_{k \rightarrow \infty} \mathrm{vol}(\Omega^c_k)=\mathrm{vol}(\Omega) \ and \ \lim_{k \rightarrow \infty} \mathcal{P}(\Omega^c_k)=\mathcal{P}(\Omega);\] 
    \item[(v)]$\Omega$ is an isoperimetric region (possibly empty) for the volume it encloses.
 \end{itemize}
\end{thm}  

We will see in the proof of Theorem \ref{ste2} below that one of the consequences of Theorem \ref{rr} is the existence of isoperimetric regions for every volume in a Riemannian manifold of finite volume as already pointed out in Remark $2.3$ of \cite{ritore1}. We also note that the condition (iv) of Theorem \ref{rr} can be expressed by saying that $\{\Omega^c_k\}_{k \in \mathbb{N}}$ \textit{converges to} $\Omega$ \textit{in the finite perimeter sense} (see \cite[pp. 4601--4603]{ritore1} or \cite{afp} for a rigorous definition). A priori we note that $\mathrm{vol}(\Omega)$ may be strictly less than $V$ in (i) of Theorem \ref{rr}. A careful analysis of the proof Theorem \ref{rr} gives a significant result of compactness, when the ambient manifold is of finite volume. This is expressed by the following corollary. %\textcolor{red}{Similar arguments are employed in the proof of Theorem $1$ of \cite{floresNardulliCompacidade} and we refer the reader to this latter paper \cite{floresNardulliContinuidade} for more details that are omitted here. Questa piccola frase andrebbe riscritta dato che ho tolto i preprints tuoi e di Abramo e ho messo semplicemente il lavoro su Geom. Dedicata...infatti lo stesso viene fatto li'... quindi andrebbero citati i teoremi in Geom. Dedicata che sono appropriat in questo punto specifico...FGR02032020}
\begin{cor}[Compactness]\label{compattezza} Let $M$ be a complete Riemannian manifold of $\mathrm{vol}(M) < \infty$. Then for any sequence $\{\Omega_k\}_{k \in \mathbb{N}}$ of sets of finite perimeter such that $\mathrm{vol}(\Omega_k)\leq V$ and $\mathcal{P}(\Omega_k) \le A$ $($where $V, A$ are positive constants$)$, there exists a set $\Omega \subseteq M$ of finite perimeter and a subsequence again noted $\{\Omega_k\}_{k \in \mathbb{N}}$ such that $\{\Omega_k\}_{k \in \mathbb{N}}$ converges to $\Omega$ in $L^1(M)$. Moreover, if $\Omega_k$ is a minimizing sequence of volume $V\in]0,\mathrm{vol}(M)[$, then we also have that the convergence is in the sense of finite perimeter sets, i.e., $\mathcal{P}(\Omega_k)\to\mathcal{P}(\Omega)$.
\end{cor}
\begin{proof}
We apply the construction of the proof of Theorem $\ref{rr}$ to the sequence $\Omega_k$ even if this is not necessarily a minimizing sequence. The conclusions are the same mutatis mutandis as those of Theorem \ref{rr} (for the details one can see the proof of Theorem $1$ of \cite{localhoelder}). Namely there exist a  finite perimeter set $\Omega$, a subsequence denoted again by $\Omega_k$, and $0<\bar{V}\leq V$, $0<\bar{A}\leq A$ such that
\begin{itemize}
 \item[(I)]$\mathrm{vol}(\Omega)=V_1 \le\bar{V}$ and $\mathcal{P}(\Omega) \le \bar{A}$;
 \item[(II)]$\mathrm{vol}(\Omega^c_k) + \mathrm{vol}(\Omega^d_k) =\bar{V}$ and \[\lim_{k \rightarrow \infty}[\mathcal{P}(\Omega^c_k) + \mathcal{P}(\Omega^d_k)]=\bar{A};\]
  \item[(III)]The sequence $\{\Omega^d_k\}_{k \in \mathbb{N}}$ diverges;
   \item[(IV)]Passing to a subsequence, we have that \[\lim_{k \rightarrow \infty} \mathrm{vol}(\Omega^c_k)=\mathrm{vol}(\Omega)=V_1 \ and \ \lim_{k \rightarrow \infty} \mathcal{P}(\Omega^c_k)\ge\mathcal{P}(\Omega).\] 
    %\item[(v)]$\Omega$ is an isoperimetric region (eventually empty) for the volume it encloses.
 \end{itemize}

Thus  there is a splitting of the volume in the following form
\[\bar{V}=\mathrm{vol}(\Omega^c_k) + \mathrm{vol}(\Omega^d_k) = \lim_{k \rightarrow \infty }\mathrm{vol}(\Omega^c_k) + \lim_{k \rightarrow \infty}\mathrm{vol}(\Omega^d_k)=V_1 + V_2 ,\] where $V_1$ is the volume which is at finite distance from $\Omega$ and $V_2$ is the volume which is at "infinite" distance from $\Omega$. Assume that $V_2>0$. By construction  $\Omega^d_k=\Omega_k-B_{r(k+1)}$ where $p_0 \in M$, is fixed once at all, $B(p_0,r)$ is the open ball centered at $p_0$ of radius $r >0$,  $B_{r(k+1)}= M \cap B(p_0,r(k+1))$.  Such $\Omega^d_k$ turns out to be a sequence that lies outside every fixed compact $K$ inside  $M$. The details  of this construction can be found at \cite[pp. 4604--4606]{ritore1}. Then it must be $\mathrm{vol}(\Omega^d_k) \le \mathrm{vol}(M-B_{r(k+1)})$. Passing through the limit, 
\[0 \le \lim_{k \rightarrow \infty}\mathrm{vol}(\Omega^d_k) \le \lim_{k \rightarrow \infty} \mathrm{vol}(M\setminus B_{r(k+1)})=0\]
and so $V_2=\lim_{k \rightarrow \infty}\mathrm{vol}(\Omega^d_k)=0$. This gives the desired contradiction. Therefore $V_2=0$, hence $V_1=\bar{V}$ and the result follows readily from (IV). 
\end{proof}

Another interesting corollary of Theorem \ref{rr} is related with the continuity of the isoperimetric profile. In order to prove this second consequence, we recall a technical lemma from \cite{ritore3}.

%\begin{lem}[Deformation's Lemma, see \cite{ritore3}, Lemma 4.5]\label{deformation}
%Let $M$ be a connected unbounded open set of a complete Riemannian manifold and $\Omega \subseteq M$ a finite perimeter set. Then there exists a finite perimeter set $\tilde{\Omega}_r \supseteq \Omega$  with $0<r<\infty$  and  a constant $C_\Omega>0$, depending only on $\Omega$, such that
 %\[\mathcal{P}(\tilde{\Omega}_r-\Omega) \le |\tilde{\Omega}_r - \Omega| \cdot C_\Omega.\]
%\end{lem}

The continuity of the isoperimetric profile is shown below.

\begin{cor}[Continuity of $I_M$]\label{continuity}
Let $M$ be a connected unbounded open set of a complete Riemannian manifold of finite volume. Then $I_M(V)$ is continuous. 
\end{cor}

\begin{proof}
Consider a sequence of volumes $V_i$ such that $V=\lim_{i \rightarrow \infty}V_i$. By Corollary \ref{compattezza}, we have that 
\begin{equation}
I_M(V)\leq\mathcal{P}(\Omega) \le  \liminf_{i \rightarrow \infty}\mathcal{P}(\Omega_i)=\liminf_{i \rightarrow \infty} I_M(V_i),
\end{equation}
where $\Omega_i$ is an isoperimetric region of $\mathrm{vol}(\Omega_i)=V_i$ and $\mathrm{vol}(\Omega)=V$ such that $\Omega_i$ converges to $\Omega$ in $L^1(M)$. In principle $\Omega$ is not necessarily an open bounded set with smooth boundary, but the first inequality is still true and \cite[Theorem 1]{localhoelder} shows that there is continuity of the isoperimetric profile. This allows us to conclude  the lower semicontinuity of $I_M(V)$. In order to show the upper semicontinuity, we need to prove that
\[I_M(V)=\mathcal{P}(\Omega) \ge  \limsup_{i \rightarrow \infty}\mathcal{P}(\Omega_i)=\limsup_{i \rightarrow \infty} I_M(V_i),\] where $\Omega$ is an isoperimetric region of volume $V$ and $\Omega_i$ is a suitable set approximating $\Omega$, but this follows from  \cite[Corollary 1, Theorem 2]{localhoelder}. 
\end{proof}

Corollary \ref{continuity} does not follows from \cite{localhoelder} because  it is proved there the continuity and  the H\"older continuity of the isoperimetric profile in Riemannian manifolds having Ricci curvature bounded below and volume or balls of a fixed radius bounded below (uniformly with respect to their centers). Manifolds of this kind are of infinite volume. So they are quite far of being the class of manifold with finite volume that we are considering in the preset paper with respect to the Hamilton isoperimetric ratios. Another difference (between the proof of Corollary \ref{continuity} and the arguments in \cite{localhoelder}) is in the way we prove the lower semicontinuity of the isoperimetric profile function. Here the lower semicontinuity is an immediate consequence of our compactness Corollary \ref{compattezza}. On the other hand, the basic idea of the proof of H\"older continuity in  \cite{localhoelder} is robust enough to be adapted to the context of manifolds with finite volume, that is, here, but we preferred to give an alternative argument because we think that it is more appropriate to the context we are studying and in some respect quite new. %So we can prove in the same way of the paper \cite{localhoelder} the continuity of the isoperimetric profile function for manifolds with finite volume.

The basic regularity properties of the boundary of isoperimetric regions are stated below.

\begin{thm}[See \cite{ritore1}, Proposition 2.4] \label{regularity}
Let $\Omega$ be an isoperimetric region in a connected open set $M$ of smooth boundary $\partial M$. Then  $\overline{\partial \Omega \cap  M}=\Sigma_r \mathring{\cup} \Sigma_s $, where $\Sigma_r$ is the regular part of $\overline{\partial \Omega \cap  M}$  and $\Sigma_s$ is the singular part of $\overline{\partial \Omega \cap  M}$.  Moreover 
\begin{itemize}
\item[(i)]$\Sigma_r \cap  M$ is a smooth embedded hypersurface with constant mean curvature; 
\item[(ii)]if $\overline{\partial \Omega \cap  M} \cap \partial M \neq \emptyset$, then $\overline{\partial \Omega \cap  M}$ meets $\partial M$  orthogonally;
\item[(iii)]$\Sigma_s$ is a closed set of Hausdorff dimension at most $n-7$.
\end{itemize}
\end{thm}

By Theorem \ref{regularity} (iii),  the isoperimetric regions have smooth boundary  in the low dimensional cases.

\section{Isoperimetric ratio in the sense of Hamilton}

In the present section we consider only complete manifolds of finite volume.
We introduce some  terminology, which can be found  in \cite{chow}, but also some new functionals for the proofs of our main results. 

\begin{defn}\label{hamilton} Let $M$ be a complete Riemannian manifold with $\mathrm{vol}(M) < \infty$. If  $N \subseteq M$ is a smooth embedded closed hypersurface which separates $M$,  we define the isoperimetric ratio 
\[I(N)=  \mathcal{P}(N)^{n+1} \cdot {\left(\frac{1}{\mathrm{vol}(M_1)}+\frac{1}{\mathrm{vol}(M_2)}\right)}^n\]
and 
\[C(N)=  \mathcal{P}(N) \cdot \left(\frac{1}{\mathrm{vol}(M_1)}+\frac{1}{\mathrm{vol}(M_2)}\right).\]
If $H \subset M$ is a smooth embedded closed (possibly disconnected) hypersurface which is the boundary of an open region $R$, we define
\[J(H)=  \mathcal{P}(H)^{n+1} \cdot {\left(\frac{1}{\mathrm{vol}(R)}+\frac{1}{\mathrm{vol}(M)-\mathrm{vol}(R)}\right)}^n\]
and
\[D(H)=  \mathcal{P}(H)\cdot \left(\frac{1}{\mathrm{vol}(R)}+\frac{1}{\mathrm{vol}(M)-\mathrm{vol}(R)}\right).\]
By default, we get four isoperimetric constants
\[I= \inf \{I(N) \ | \ N \ \mathrm{separates} \ M\} , \ \ \  \  C=\inf \{C(N) \ | \ N \ \mathrm{separates} \ M\}, \]
\[J= \inf \{J(H) \ | \ H \  \mathrm{is} \  \mathrm{smooth} \ \}, \ \  \ \ D=\inf \{D(H) \ | \ H \ \mathrm{is} \  \mathrm{smooth} \ \},\]
and  four functionals
\[\tilde{I}^*_M(V)={(\tilde{I}_M(V))}^{n+1} \cdot {\left(\frac{1}{V}+\frac{1}{\mathrm{vol}(M)-V}\right)}^n, \]
\[\tilde{I}^\sharp_M(V)=\tilde{I}_M(V) \cdot \left(\frac{1}{V}+\frac{1}{\mathrm{vol}(M)-V}\right), \]
\[I^*_M(V)={(I_M(V))}^{n+1} \cdot {\left(\frac{1}{V}+\frac{1}{\mathrm{vol}(M)-V}\right)}^n ,\]
\[I^\flat_M(V)= I_M(V) \cdot \left(\frac{1}{V}+\frac{1}{\mathrm{vol}(M)-V}\right) ,\]
which lead to the isoperimetric constants 
\[\tilde{I}^*= \inf \{\tilde{I}^*_M(V) \ | \ V \  \in \ ]0,\mathrm{vol}(M)[ \  \}, \] 
\[\tilde{I}^\sharp= \inf \{\tilde{I}^\sharp_M(V) \ | \ V \  \in \ ]0,\mathrm{vol}(M)[ \  \}, \] 
\[I^*= \inf \{I^*_M(V) \ | \ V \  \in \ ]0, \mathrm{vol}(M)[ \  \}, \] 
\[I^\flat= \inf \{I^\flat_M(V) \ | \ V \  \in \ ]0, \mathrm{vol}(M)[ \  \}. \] 
\end{defn}

\medskip
\medskip
\medskip
\medskip

In particular, if  $(M,g)$ is isometric to  $(\mathbb{R}^2, g)$ with a complete Riemannian metric $g$, we may specialize  $I(N)$, writing $n=1$ and $N=\gamma$, which turns out to be a   closed simple curve of $\mathbb{R}^2$ of length $L(\gamma)$, and let $A_1(\gamma)$  and $A_2(\gamma)$ denote the areas of the regions inside and outside $\gamma$ respectively. In this way, we get $I( \gamma )= L(\gamma)^2 \ \cdot (1/A_1(\gamma) +  1/A_2(\gamma))$ and $I=\inf \{ I(\gamma) \ | \ \gamma \ \mathrm{separates} \ \mathbb{R}^2 \}$. This special case presents some peculiarities and was studied in \cite{hamilton1, hamilton2}. We will focus on it  in Examples \ref{optimal} and \ref{optimalbis}.

Now we begin to analyze some problems of minimization of  the functionals $I(N)$, $C(N)$, $J(H)$ and $D(H)$,  in Definition \ref{hamilton}. These are not all equivalent, mainly for two reasons. A first reason is of topological nature. When we go to minimize over separating hypersurfaces, the topology and the metric of the manifold  influence strongly our arguments of proof. A second reason is due to the analytic expressions of $I(N)$, $C(N)$, $J(H)$ and $D(H)$. For instance, we note that the multiplicative factor is linear only in $C(N)$ and $D(H)$, while this is no longer true in $I(N)$ and $J(H)$. This gives complications and forces us to use some different techniques  of proof. The first case  concerns $J(H)$; this is an easy observation.

 \begin{rem}\label{ste1} Let $M$ be a complete Riemannian manifold of dimension $\geq 2$ of finite volume. With the notations of Definition \ref{hamilton}, with a suitable decay of the area of big geodesic balls we have that $J=I^*=0$. In fact, we evaluate $J(\partial B(p,r))$, where $B(p,r)$ is the geodesic ball centered at $p \in M$ of radius $r >0$. Now, on one hand, $\lim_{r \rightarrow \infty} \mathrm{vol}(B(p,r))=\mathrm{vol}(M)$, but on another hand from the coarea formula (see \cite[Theorem VIII.3.3]{chavel}) we have that 
  $$\liminf_{r \rightarrow \infty} \mathcal{P}(\partial B(p,r))=0.$$ To see this it is enough to note that $V(M)=\int_0^{+\infty}\mathcal{P}(\partial B(p,s))ds<+\infty$.  Now we can easily conclude that 
 \begin{eqnarray*}
 J & \le & \lim_{r \rightarrow \infty}J(\partial B(p,r))\\
  & = & \lim_{r \rightarrow \infty}(\mathcal{P}(\partial B(p,r)))^{n+1} \cdot {\left(\frac{1}{\mathrm{vol}(B(p,r))}+\frac{1}{\mathrm{vol}(M)-\mathrm{vol}(B(p,r))}\right)}^n\\
  & = & \lim_{r \rightarrow \infty} \frac{(\mathcal{P}(\partial B(p,r)))^{n+1}}{{(\mathrm{vol}(M)-\mathrm{vol}(B(p,r)))}^n} \cdot {\left(\frac{\mathrm{vol}(M)-\mathrm{vol}( B(p,r))}{\mathrm{vol}( B(p,r))}+1\right)}^n\\      
 & = & \lim_{r \rightarrow \infty} \frac{(\mathcal{P}(\partial B(p,r)))^{n+1}}{{(\mathrm{vol}(M)-\mathrm{vol}(B(p,r)))}^n}=0,
 \end{eqnarray*}
provided $\mathcal{P}(\partial B(p,r))\sim\frac1{r^{1+\varepsilon}}$. The same computation shows that for the metrics considered in \cite[Theorem 1.1 and Proposition 1.2]{hsu} the isoperimetric ratios $I$ and $J$ are equal to zero. In fact with the same notations of   Theorem \ref{hsu1} below, set $A(0_{\mathbb{R}^2},r)$ be the coordinate ball $\{x\in\mathbb{R}^2:|x| \le r\}$ we have that 
\begin{eqnarray*}
J & \le & \lim_{r \rightarrow \infty}J(\partial A(0_{\mathbb{R}^2},r))\\
   & = & \lim_{r \rightarrow \infty}\lambda_2(r)^{\frac{n+1}2} \cdot {\left(\frac{1}{\mathrm{vol}(A(0_{\mathbb{R}^2},r))}+\frac{1}{\mathrm{vol}(M)-\mathrm{vol}(A(0_{\mathbb{R}^2},r))}\right)}^n\\
 & = & \lim_{r \rightarrow \infty} \frac{\lambda_2(r)^{\frac{n+1}2}}{{(\mathrm{vol}(M)-\mathrm{vol}(A(0_{\mathbb{R}^2},r)))}^n} \cdot {\left(\frac{\mathrm{vol}(M)-\mathrm{vol}(A(0_{\mathbb{R}^2},r))}{\mathrm{vol}(A(p,r))}+1\right)}^n\\      
& = & \lim_{r \rightarrow \infty} \frac{\lambda_2(r)^{\frac{n+1}2}}{{(\mathrm{vol}(M)-\mathrm{vol}(A(0_{\mathbb{R}^2},r)))}^n}\le \lim_{r \rightarrow \infty}\frac{\lambda_2(r)^{\frac{n+1}2}}{\left(\int_r^{+\infty}\sqrt{\lambda_1(r)}dr\right)^n}\\
 & \le & \lim_{r \rightarrow \infty}\frac{\lambda_1(c_0r)^{\frac{n+1}2}}{\delta c_0^nr^n\lambda_1(r+\theta)^{\frac n2}}\le\lim_{r \rightarrow \infty}\frac{\lambda_1(c_0r)^{\frac12}}{\delta c_0^nr^n}=0,
 \end{eqnarray*}
with $r<r+\theta<c_0r$. Note that the inequality just before the final one follows from the assumption \eqref{Eq:Hsu(1.5)} of Theorem \ref{hsu1} below. A fortiori $J=0$. On the other hand the definitions show that  $0 \le I^* \le J$ and so $I^*=J=0$. 
This justifies the fact of considering the isoperimetric ratios $C$ and $D$ in the paper \cite{hsu} instead of the isoperimetric ratios $I$ and $J$ in dimension $2$ (i.e., with $n=1$) considered earlier by Hamilton in the case of compact manifolds. 
\end{rem} 
In the next proposition we show the equivalence of the isoperimetric ratios variational problem under investigation in dimension $3$ or higher.      
      
\begin{prop}\label{Prop:EquivalenceInHigherDimensions0}
Any Riemannian manifold $M$ of dimension $n + 1 \ge 3$ and of  finite volume satisfies $J=I$ and $C=D$.
\end{prop}

\begin{proof} By Definition \ref{hamilton}, we get  $J\le I$.
Suppose that one has a minimizing sequence (i.e., $\mathcal{P}(\Omega_j)\to J$) of connected $\Omega_j\subset M$ for every $j$, in dimension dimension  $n+1 \ge 3$ and assume that $M\setminus\Omega_j$ is not connected (See Fig. 1) then one can consider the construction of a separating competitor schematically represented in Fig. 2. 

\bigskip

\begin{figure}[h!]
 \includegraphics[scale=0.8]{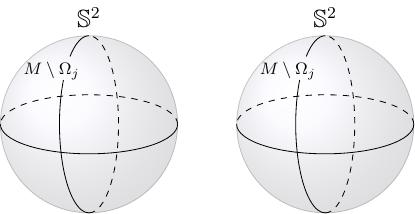}
    \caption{An example in dimension $n+1\ge3$ of $\Omega_j$ connected and $M\setminus\Omega_j$ disconnected as in the proof of Proposition \ref{Prop:EquivalenceInHigherDimensions0}.}
\end{figure}

\bigskip

\begin{figure}[h!]
 \includegraphics[scale=1]{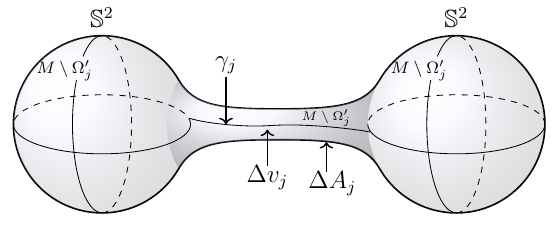}
    \caption{An example in dimension $n+1\ge3$ of the competitor $\Omega'_j$ constructed from $\Omega_j$ in the proof of Proposition \ref{Prop:EquivalenceInHigherDimensions0}.}
\end{figure}

Take a curve $\gamma_j:[0,1]\to M$ such that $\gamma_j(]0,1[)\subset\Omega_j$ joining two distinct points in the boundary of $\gamma_j(0),\gamma_j(1)\in\partial\Omega_j$, $\gamma_j(0)\ne\gamma_j(1)$, then take a small tubular neighborhood of $\gamma_j$ of small area $\Delta A_j\to0$ and small volume $\Delta V_j\to0$ and glue it smoothly to $\partial\Omega_j$ and one obtain a connected competitor $\Omega'_j$, such that $M\setminus\Omega_j$ is connected and $|\mathcal{P}(\Omega_j)-\mathcal{P}(\Omega'_j)|\to0$. So $J=\lim_{j\to+\infty}\mathcal{P}(\Omega'_j)\ge I$. To prove that $C=D$ one can proceed analogously.
\end{proof}

What can we say about $\tilde{I}^\sharp$ and $C$ ? The previous argument of Remark \ref{ste1} cannot be applied and  we are in need of a new proof.

\begin{thm}\label{ste2}
Let $M$ be a complete Riemannian manifold of dimension $\geq 2$ of finite volume. With the notations of Definition \ref{hamilton}, we have that $I^\flat \le \tilde{I}^\sharp \le C$. 
\end{thm}

\begin{proof}From Theorem \ref{rr},  for every volume $V$ there always exists an isoperimetric region $\Omega$ with $\mathrm{vol}(\Omega)=V$.   Theorem \ref{regularity} implies that $\partial\Omega$ is  smooth in low dimensions (i.e. $n +1 \le 7$). In higher dimensions there is a sequence $\{\Omega_i\}_{i \in \mathbb{N}}$ of regions of finite perimeter with smooth boundaries $\{\partial \Omega_i\}_{i \in \mathbb{N}}$   converging to $\partial \Omega$ (see \cite[Proposition 1.4]{miranda}).    Now if  $N \subseteq M$ is a smooth embedded closed (possibly disconnected) hypersurface which separates $M$,  then
\[C(N)  =  \mathcal{P}(N)\cdot {\left(\frac{1}{\mathrm{vol}(M_1)}+\frac{1}{\mathrm{vol}(M_2)}\right)}\] 
\[ \geq  {\tilde{I}_M(V)} \cdot \left(\frac{1}{V}+\frac{1}{\mathrm{vol}(M)-V}\right)=\tilde{I}^\sharp_M(V) \geq I_M(V) \cdot \left(\frac{1}{V}+\frac{1}{\mathrm{vol}(M)-V}\right)=I^\flat_M(V).\]
Passing through the infimums, we get
\[C=\inf\{C(N) \ | \ N \ \mathrm{separates} \ M\}\]
\[ \ge \inf\{\tilde{I}^\sharp_M(V) \ | \ V \ \in \ ]0, \mathrm{vol}(M)[ \ \}=\tilde{I}^\sharp \ge I^\flat.\]
\end{proof}

A priori $I^\flat$ may be zero or not. We will give more details on this point in the next section. Now we prove a similar result  for $I^*$, $\tilde{I}^*$ and $I$. 

\begin{thm}\label{ste3}
Let $M$ be a complete Riemannian manifold of dimension $\geq 2$ of finite volume. With the notations of Definition \ref{hamilton}, we have that $I^* \le \tilde{I}^* \le I$. % and $I^*=0$. 
\end{thm}

\begin{proof}We may argue as in Theorem \ref{ste2}, in order to show $I^* \le \tilde{I}^* \le I$. %\textcolor{red}{It remains to check that $I^*=0$. Now there exists a ball $B(p,r)$  at $p \in M$ of radius $r >0$ such that $\lim_{r \rightarrow \infty} \mathrm{vol}(B(p,r))=\mathrm{vol}(M)$ and $\lim_{r \rightarrow \infty} \mathcal{P}(B(p,r))=0$ by the  coarea formula (see \cite[Theorem VIII.3.3]{chavel}). The same  argument of Remark \ref{ste1} implies$I^* = 0$.}
\end{proof}

A final observation concerns $D$ and it is  easy to check.

\begin{rem}\label{d} With the notations of Definition \ref{hamilton}, we have $I^\flat \le D \le C$.
\end{rem}

\section{Minimization problems }

An interesting question is to know whether the inequalities in Theorems \ref{ste2} and \ref{ste3} become equalities or not. An answer to this question may depend on the topology of the ambient manifold $M$ and on the dimension of $M$. We will investigate such aspects in the present section,  beginning with two useful lemmas which provide information on the number of connected components of the regions whose boundary minimize $C$ (in the sense of Definition \ref{hamilton}).  We apply an argument of algebraic nature, which is inspired by  \cite[Lemma 5.86 ]{chow}, and could be found also in \cite[Lemma 2.7]{hsu}

\begin{lem}[First Condition of Superadditivity]\label{algebraic}
For every positive real numbers $L_1$, $L_2$, $A_1$, $A_2$, $A_3$, we get
%\tiny
%\begin{equation}\label{Eq:algebraicStatement}
\[(L_1+L_2)\left(\frac{1}{A_1+A_2}+\frac{1}{A_3}\right)>\min\left\{L_1\left(\frac{1}{A_1}+\frac{1}{A_2+A_3}\right), L_2\left(\frac{1}{A_2}+\frac{1}{A_1+A_3}\right)\right\}.\]
%\end{equation}
%\normalsize
\end{lem}

\begin{proof} By contradiction,  assume that $L_1, L_2, A_1, A_2, A_3 $ satisfy
\[ (L_1+L_2)\left(\frac{1}{A_1+A_2}+\frac{1}{A_3}\right)\leq L_1\left(\frac{1}{A_1}+\frac{1}{A_2+A_3}\right),\]
\[ (L_1+L_2)\left(\frac{1}{A_1+A_2}+\frac{1}{A_3}\right)\leq L_2\left(\frac{1}{A_2}+\frac{1}{A_1+A_3}\right).\]
We rewrite the two preceding inequalities respectively as
\[\frac{A_1(A_2+A_3)}{A_3(A_1+A_2)}\leq\frac{L_1}{L_1+L_2} \ \ \mathrm{and} \ \ \frac{A_2(A_1+A_3)}{A_3(A_1+A_2)}\leq\frac{L_2}{L_1+L_2}.\]
Summing up these  two last inequalities, we get
\[\frac{A_3(A_1+A_2)+2A_1A_2}{A_3(A_1+A_2)}=1+\frac{2A_1A_2}{A_3(A_1+A_2)}\leq\frac{L_1+L_2}{L_1+L_2}=1,\]
which imply
\[\frac{2A_1A_2}{A_3(A_1+A_2)}\leq 0.\]
This gives a contradiction, because we assumed $A_1$, $A_2$ and $A_3$ strictly positive.
\end{proof}

The use of Lemma \ref{algebraic} is to argue properties of the topological nature of minimisers. This will be more clear in the following result.

\begin{lem}[Separating property for minimisers]\label{connection}Let $M$ be a complete Riemannian manifold of dimension $\geq 2$ of finite volume and
let  $\Omega$ be a finite perimeter set in  $M$ that minimizes $I^\flat$, i.e., 
\[D(\partial\Omega)=\mathcal{P}( \Omega)\left(\frac{1}{\mathrm{vol}(\Omega)}+\frac{1}{\mathrm{vol}(M)-\mathrm{vol}(\Omega)}\right)=I^\flat.\]
Then $\Omega$ and $M-\Omega$ are connected,  $\partial\Omega$ separates  $M$ and $I^\flat=C=D$. In particular, if $n+1\leq 7$, 
then $\partial\Omega$ is a smooth hypersurface that separates $M$.
\end{lem}
\begin{proof}
The proof goes by contradiction. Firstly, we show that $\Omega$ is connected. In order to do this, we suppose that $\Omega=\Omega_1\mathring{\cup}\Omega_2$ contains two connected components $\Omega_1$ and $\Omega_2$ such that $\mathrm{vol}(\Omega_1)=A_1$, 
$\mathrm{vol}(\Omega_2)=A_2$, $A_3=\mathrm{vol}(M)-(A_1+A_2)$, $L_1=\mathcal{P}(\Omega_1)$, and $L_2=\mathcal{P}(\Omega_2)$. Then 
\[D(\partial\Omega)=(L_1+L_2)\left(\frac{1}{A_1+A_2}+\frac{1}{A_3}\right),\] 
\[ D(\partial\Omega_1)=L_1\left(\frac{1}{A_1}+\frac{1}{A_2+A_3}\right),\] 
\[ D(\partial\Omega_2)=L_2\left(\frac{1}{A_2}+\frac{1}{A_1+A_3}\right).\] 
Applying Lemma \ref{algebraic}, we find 
\[D(\partial\Omega)> \min\left\{D(\partial\Omega_1), D(\partial\Omega_2)\right\} \geq \min\left\{I^\flat_M(A_1), I^\flat_M(A_2)\right\},\] which contradicts the minimality of $\Omega$. We conclude that $\Omega$ must be connected. Now  $\mathcal{P}(\Omega)=\mathcal{P}(M-\Omega)$ so the same argument implies that $M-\Omega$ is connected. This implies  that $\partial\Omega$  separates $M$. 

Then $C(\partial \Omega)=D(\partial \Omega)=I^\flat$ implies that $C \le D=I^\flat$. On the other hand, Remark \ref{d} shows $C \ge D=I^\flat$ and so $C=D=I^\flat$.
The remaining part of the result for the low dimensions follows from Theorem \ref{regularity}.
\end{proof}

We have all the ingredients for the proof of one of our main results. In connection with Theorem \ref{ste4} there are in the literature preceding results, in the case where $M=\mathbb{R}^2$ and the metric $g$ satisfies very special conditions, namely Theorem $1.1$ of \cite{hsu}. We make here a comparison between our Theorem \ref{ste4} and Theorem $1.1$ of \cite{hsu}. For  sake of completeness we report the result below.

\begin{thm}[See \cite{hsu}, Theorem 1.1]\label{hsu1}Let \(g=\left(g_{i j}\right)\) be a complete Riemannian metric on \(\mathbb{R}^{2}\) with finite total area \(A=\)
\(\int_{\mathbb{R}^{2}} d V_{g}\) satisfying
\begin{equation}\label{Eq:Hsu(1.1)}
\lambda_{1}(|x|) \delta_{i j} \leq g_{i j}(x) \leq \lambda_{2}(|x|) \delta_{i j}, \forall|x| \geq r_{0}    
\end{equation}
for some constant $r_{0}>1$ and positive monotone decreasing functions  $\lambda_{1}(r), \lambda_{2}(r)$ on $\left[r_{0}, \infty\right)$ that satisfy
\begin{equation}\label{Eq:Hsu(1.2)}
  \int_{r}^{c_{0} r} \sqrt{\lambda_{1}(\rho)} d \rho \geq \pi r \sqrt{\lambda_{2}(r)},\quad \forall r \geq r_{0},  
\end{equation}
\begin{equation}\label{Eq:Hsu(1.3)}
   r \sqrt{\lambda_{1}\left(c_{0} r\right)} \geq b_{1} \int_{r}^{\infty} \rho \lambda_{2}(\rho) d \rho, \quad \forall r \geq r_{0} ,
\end{equation}
\begin{equation}\label{Eq:Hsu(1.4)}
    \int_{r}^{r^{2}} \sqrt{\lambda_{1}(\rho)} d \rho \geq b_{2}, \quad \forall r \geq r_{0},
\end{equation}
\begin{equation}\label{Eq:Hsu(1.5)}
 \lambda_{1}\left(c_{0} r\right) \geq \delta \lambda_{2}(r), \quad \forall r \geq r_{0},   
\end{equation}
 for some constants \(c_{0}>1, b_{1}>0, b_{2}>0, \delta>0,\) where \(|x|\) is the distance of \(x\) from
the origin with respect to the Euclidean metric. For any closed simple curve \(\gamma\) in \(\mathbb{R}^{2}\),
let (cf. [2])

$$
I(\gamma)=L(\gamma)\left(\frac{1}{A_{\text {in }}(\gamma)}+\frac{1}{A_{\text {out }}(\gamma)}\right)
$$
where \(L(\gamma)\) is the length of the curve \(\gamma\), \(A_{\text {in}}(\gamma)\) and \(A_{\text {out}}(\gamma)\) are the areas of the
regions inside and outside \(\gamma\) respectively, with respect to the metric \(g\). Let
$$
I=I_{g}=\inf _{\gamma} I(\gamma)
$$
where the infimum is over all closed simple curves \(\gamma\) in \(\mathbb{R}^{2}\).
Suppose \(g\) satisfies \eqref{Eq:Hsu(1.1)} for some constant \(r_{0}>1\) where \(\lambda_{1}(r), \lambda_{2}(r)\)
are positive monotone decreasing functions on \(\left[r_{0}, \infty\right)\) that satisfy \eqref{Eq:Hsu(1.2)},\eqref{Eq:Hsu(1.3)},\eqref{Eq:Hsu(1.4)}
and \eqref{Eq:Hsu(1.5)} for some constants \(c_{0}>1, b_{1}>0, b_{2}>0\) and \(\delta>0 .\) Then there exists \(a\)
constant \(b_{0}>0\) depending on \(b_{1}, b_{2}\) and \(A\) such that the following holds:
If $I_{g}<b_{0}$, then there exists a closed simple curve \(\gamma\) in \(\mathbb{R}^{2}\) such that \(I_{g}=I(\gamma)\). Hence \(I_{g}>0\).
\end{thm}

\begin{rem}\label{Rem:ComparisonWithHsu}
In  \cite[Lemma 2.1]{hsu} it is showed that under the assumptions of Theorem \ref{hsu1} there is a minimizing sequence of separating Jordan curves $(\gamma_j)_{j\in\mathbb{R}}$ that stay inside a compact set of $\mathbb{R}^2$. This means that \[{\underset{V \rightarrow 0} \liminf \ \frac{I_M(V)}{V}} \ =+\infty>C_1=b_0>0,\] because in a compact manifold by Berard-Meyer \cite[Appendix C]{beme}  we have that $I_M(V)\sim c_2V^{\frac12}$, where $c_2>0$ is the $2$-dimensional Euclidean isoperimetric constant i.e., $c_2=2\sqrt{\pi}$ and $b_0>0$  is the constant found in Theorem \ref{hsu1}. The meaning of $I_g$ in  \cite[Lemma 2.1]{hsu} is exactly the meaning of $\tilde{I}^\sharp$ here. Note that $\tilde{I}^\sharp\ge I^\flat$ and that in the plane our notion of separating curves coincides with  the notion of simple closed curves  in \ref{hsu1}. This is clear by  the results of \cite{Polulyakh}. However there are examples showing that $\tilde{I}_M\ne I_M$. To see this in the case in which $M$ is compact one can consult Section $9$, page $485$, Fig. $2$ of \cite{BavardPansu}. It is proved there that there are isoperimetric regions in a $2$ dimensional Riemannian manifold such that the complementary set is disconnected. Starting with the example constructed by C. Bavard and P. Pansu \cite{BavardPansu} one can easily construct a complete non-compact Riemannian manifold of dimension $2$ and of finite volume $(M^2,g)$, in which there are isoperimetric regions $\Omega$ such that $M\setminus\Omega$ is disconnected.
\end{rem}

Looking at Remark \ref{Rem:ComparisonWithHsu},  the hypothesis of  the following theorem are satisfied when $C_1=b_0$. On the other hand, our result applies to much more general Riemannian manifolds and do not depend directly on the dimension two.

 Notice also that the case of $M$  compact (in Theorem \ref{ste4}) is known in  literature, because it follows by results of  Hamilton \cite{hamilton2} (see also  \cite[Lemma 5.82]{chow}). The real contribution of the theorem below deals with the noncompact case $M$ and with the formulation of the conditions  $(i)$ and $(ii)$. %because ${\underset{V \rightarrow 0} \liminf \ \frac{I_M(V)}{V}}=+\infty$ and by classical compactness and superadditivity a minimizer for $I^\flat$ (i.e., satisfying the conclusions of Theorem \ref{ste4}) always exists.

\begin{thm} \label{ste4} Let $M$ be a complete Riemannian manifold of finite volume of dimension $n+1 \ge 2$. If $M$ satisfies  the following two conditions for a positive constant $C_1:$
\begin{itemize}
\item[(i)] ${\underset{V \rightarrow 0} \liminf \ \frac{I_M(V)}{V}}\geq C_1>0$;
 \item[(ii)]$I^\flat<C_1;$
\end{itemize}
then  there exists a connected isoperimetric region $\Omega \subseteq M$ of positive volume such that $M\setminus\Omega$ is connected and 
 \begin{equation}\label{Eq:Ste4Statement}
 \tilde{I}^\sharp=I^\flat=C(\partial \Omega)=D(\partial \Omega) =C= D>0.
 \end{equation}  
 In particular, $C=D>0$ and, if $n+ 1 \leq 7$,  $\partial \Omega$  separates $M$ and is smooth. Conversely, if  $M$ is noncompact and $\Omega \subseteq M$ is a connected unbounded finite perimeter set of positive volume such that $M\setminus\Omega$ is connected satisfying \eqref{Eq:Ste4Statement}, then $\Omega$ is an isoperimetric region and both $(i)$ and $(ii)$ hold.
 \end{thm}
 
      \begin{proof} We start proving the first implication. By Corollary \ref{continuity} the function $$ V \in ]0,\mathrm{vol}(M)[ \mapsto I^\flat_M(V) \in \ ]0,\infty [$$ is continuous on $]0, \mathrm{vol}(M)[$. Conditions (i) and (ii) ensure that this function attains its infimum at a global minimum point $V_0\in ]0, \mathrm{vol}(M)[$, i.e., there exists $V_0 \in ]0, \mathrm{vol}(M)[$ such that $I^\flat_M(V_0)=I^\flat$. By Theorem \ref{rr}, there exists an isoperimetric region $\Omega$ of volume $\mathrm{vol}(\Omega)=V_0$  such that $I^\flat_M(\mathrm{vol}(\Omega))=D(\partial \Omega)$. Lemma \ref{connection} shows that $\Omega$ and $M-\Omega$ are connected.  Theorem \ref{regularity} shows that  $\partial\Omega$ is smooth, if $n+1\leq 7$. 
      
      Conversely, we prove the reverse implication. First of all,  note that the assumptions on $\Omega$ imply that it is an isoperimetric region. Now we argue by contradiction.     If $$I^\flat\ge{\underset{V \rightarrow 0^+} \liminf \ \frac{I_M(V)}{V}},$$ then $0\le{\underset{V \rightarrow 0} \liminf \ I_M(V)}\le\underset{V \rightarrow 0} \lim I^\flat V$. It follows that $$\underset{V \rightarrow 0^+} \liminf \ I_M(V)\left[\frac{1}{V}+\frac{1}{\mathrm{vol}(M)-V}\right]=\underset{V \rightarrow 0^+} \liminf \ \frac{I_M(V)}{V}\le I^\flat.$$ 
      On the other hand, by definition we get $$I^\flat\le\underset{V \rightarrow 0^+} \liminf \ I_M(V)\left[\frac{1}{V}+\frac{1}{\mathrm{vol}(M)-V}\right].$$ Thus 
      \begin{equation}\label{Eq:Ste4}
      I^\flat=\underset{V \rightarrow 0^+} \liminf \ I_M(V)\left[\frac{1}{V}+\frac{1}{\mathrm{vol}(M)-V}\right]=C(\partial\Omega). 
      \end{equation}
      Notice that $I$ could have multiple positive minima, so the proof is more complicated.
      In first we observe that if $M$ is compact $(i)$ and $(ii)$ are automatically satisfied for every $C_1>I^\flat$, because ${I_M(v)}{v}\to+\infty$ as $v\to0^+$. We divide the remaining of the proof into two parts. In the first one we assume that $M$ and $\Omega$ are unbounded and in the second part we assume that $\Omega$ is bounded while keeping $M$ unbounded. With this aim in mind, assume that $\Omega$ is unbounded. Fix a point $p$ in $M$ and consider a sequence of radii $R_j$ for $j \in \mathbb{N}$ getting to $\infty$ increasily, set $\Delta V_j:=\mathrm{vol}(\Omega\setminus B(p, R_j))$ then by coarea formula we can find a radius $\tilde{R}_j$ such that $R_j\le\tilde{R}_j\le 2R_j$ and $A(\Omega\cap\partial B_g(p, \tilde{R}_j))\le\frac{\Delta V_j}{R_j}$ as in the proof of Corollary \ref{compattezza}. Define the sequence of competitors for the isoperimetric problem \[\Omega_j:=\Omega\setminus B_g(p,\tilde{R}_j) \mathring{\cup} E_j,\] as in Figure 3 below,
    \begin{figure}[h!]\label{Fig:IfPartUnbounded}
 \includegraphics[scale=0.7]{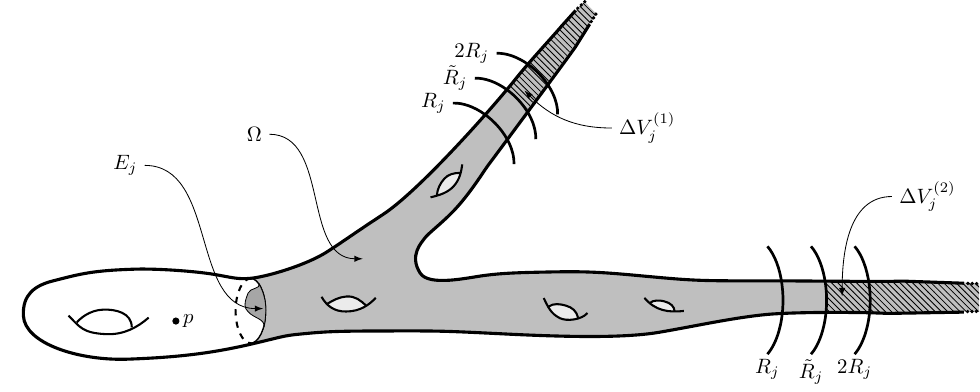}
    \caption{Construction of a competitor, showing that one cannot have unbounded absolute minimizers, if $(i)$ and $(ii)$ are simultaneously satisfied. Here $\Delta V_j:=\Delta V^1_j+\Delta V^2_j$.}
\end{figure}  
\normalsize
      where $E_j$ is a small perturbation of the boundary of $\Omega$ in a neighborhood of a fixed point $x\in\partial\Omega$ (obtained as in the deformation Lemma  \cite[Lemma 3.10]{nardulli2}) and satisfying $\mathrm{vol}(E_j)=\Delta V_j$. It is easy to check that $\mathrm{vol}(\Omega_j)=\mathrm{vol}(\Omega)$ for every $j\in\mathbb{N}$. Since $\Omega$ is an isoperimetric region, we denote by $H(\partial\Omega)$ scalar product of the mean curvature vector of $\partial\Omega$ with the inward pointing normal at a regular boundary point and get by \cite[Lemma 3.10]{nardulli2}
      \begin{equation}\label{Eq:CompetitorsIfPart}
      0\le A(\partial\Omega_j)-A(\partial\Omega)\le H(\partial\Omega)(1+\varepsilon_j)\Delta V_j+2\frac{\Delta V_j}{R_j}-A(\partial(\Omega\cap (M\setminus B_g(p, \tilde{R}_j))),
      \end{equation}
      where $\varepsilon_j\to0$.
      Dividing the preceding equation by $\Delta V_j$ we have 
      \begin{equation}\label{Eq:CompetitorsIfPart0}
      \frac{A(\partial(\Omega\cap (M\setminus B_g(p, \tilde{R}_j)))}{\Delta V_j}\le H(\partial\Omega)(1+\varepsilon_j)+\frac2{R_j}.
      \end{equation} 
      Finally taking the limit as $j\to+\infty$, and recalling that the left hand side of \eqref{Eq:CompetitorsIfPart0}  tends to $I^{\flat}$ by \eqref{Eq:Ste4}, we get
      \begin{equation}\label{Eq:CompetitorsIfPart1}
      I^{\flat}=\lim_{j\to\infty}\frac{A(\partial(\Omega\cap (M\setminus B_g(p, \tilde{R}_j)))}{\Delta V_j}\le\lim_{j\to\infty}H(\partial\Omega)(1+\varepsilon_j)+\frac2{R_j}=H(\partial\Omega).
      \end{equation}
       Notice that this construction is always possible, because $\Omega$ is an isoperimetric region and so by regularity theory the boundary of $\Omega$ is an open set (actually dense). Now we recall the argument of Hamilton in \cite{hamilton4} with equidistant variations of $\Omega$ that allow us to determine the exact value $H(\partial\Omega)$, i.e., 
       \begin{equation}\label{Eq:HamiltonVariationalEquality}
       H({\partial\Omega})=\mathcal{P}(\Omega)\left[\frac1{\mathrm{vol}(\Omega)}-\frac1{\mathrm{vol}(M)-\mathrm{vol}(\Omega)}\right].
       \end{equation}
      Equation \eqref{Eq:HamiltonVariationalEquality} is obtained considering the equidistant domains $\Omega_t$. In dimension $2$  $\Omega_t$ they are still connected with connected complement for $t$ small enough, this is easy to see when $\partial\Omega$ is smooth. In dimension $3$ and higher to apply the argument of Hamilton we do not need to be sure that the equidistant $\Omega_t$ are still separating for small $t$, because of the equivalence with the non separating problem showed in Proposition \ref{Prop:EquivalenceInHigherDimensions0}. Now, deriving the Hamilton functional (restricted to the equidistant domains) \[\varphi : t \in [0,\varepsilon] \mapsto \varphi(t)=\mathcal{P}( \Omega_t)\left(\frac{1}{\mathrm{vol}(\Omega_t)}+\frac{1}{\mathrm{vol}(M)-\mathrm{vol}(\Omega_t)}\right) \in [0, + \infty[\] defined  for some $\varepsilon>0$ and evaluating in $t=0$. Then using the minimality of $\Omega_0=\Omega$ we get the condition $\varphi'(0)=0$, carrying out the due computations one easily obtains \eqref{Eq:HamiltonVariationalEquality}.
        Combining \eqref{Eq:Ste4Statement}, \eqref{Eq:CompetitorsIfPart1}, \eqref{Eq:HamiltonVariationalEquality} we conclude that
      \begin{eqnarray*}
      \mathcal{P}(\Omega)\left[\frac{1}{\mathrm{vol}(\Omega)}+\frac{1}{\mathrm{vol}(M)-\mathrm{vol}(\Omega)}\right] &  \stackrel{\eqref{Eq:Ste4Statement}}{=} & I^{\flat} \\  &  \stackrel{\eqref{Eq:CompetitorsIfPart1}}{\le} &  H({\partial\Omega})\\ & \stackrel{\eqref{Eq:HamiltonVariationalEquality}}{=} & \mathcal{P}(\Omega)\left[\frac{1}{\mathrm{vol}(\Omega)}-\frac{1}{\mathrm{vol}(M)-\mathrm{vol}(\Omega)}\right],
      \end{eqnarray*} which is the desired contradiction implying that the unbounded positive volume minimizer does not exists. Therefore the result follows.       
      \end{proof}

\begin{rem}\label{Rem:CounterExSte4}
In the proof of the previous theorem, we cannot avoid the assumption of having unbounded minimizers, because it is not too hard to construct examples of manifolds admitting compact isoperimetric regions minimizing $I$ and not satisfying $(i)$ and $(ii)$ of Theorem \ref{ste4}. Roughly speaking take a compact manifold and fix $\Omega$ as a solution of our minimization problem. Then in  $M \setminus \Omega$ make a surgery and attach a finite volume tail such that the complement of large geodesic balls are isoperimetric regions with the ratio area/volume tending to $I^\flat$.   
\end{rem}

We may replace (i) of Theorem \ref{ste4} with another condition.

\begin{lem}[Symmetry's Lemma]\label{equivcond}Let $M$ be a complete Riemannian manifold of finite volume $A=\mathrm{vol}(M)$  and $C_1$ a positive constant. Then
the following statement are equivalent:
\begin{itemize}
\item[(j)] ${\underset{V \rightarrow 0} \liminf \ \frac{I_M(V)}{V}} \ \geq C_1>0$; 
\item[(jj)]${\underset{V \rightarrow A} \liminf \ \frac{I_M(V)}{A-V}} \ \geq C_1>0$.
\end{itemize}
\end{lem}

\begin{proof}It is enough to note that $I_M(V)=I_M(A-V)$ for all $V \in \ ]0,A[$. The rest is just an application of the definitions.
\end{proof}

Therefore Theorem \ref{ste4} may be reformulated.

\begin{cor}\label{ste4bis} Theorem \ref{ste4} is true when we replace (i) with (jj) of Lemma \ref{equivcond}.
\end{cor}

\begin{proof} It follows from Lemma \ref{equivcond}.
\end{proof}

If we want to formulate an analogous result of Theorem \ref{ste4} for $J(H)$,  we have problems with the condition (i) of Theorem \ref{ste4}, since it might happen that \[\liminf_{V \rightarrow 0} \ \frac{I_M(V)^{n+1}}{V^n}=0. \]     
About the functional $I(N)$, the limit
 \[\liminf_{V \rightarrow 0} \ \frac{\tilde{I}_M(V)^{n+1}}{V^n}\]
 may be zero or not, but the previous arguments shall be modified. %This is illustrated in the following theorem
%\begin{thm} \label{ste5} Let $M$ be a complete Riemannian manifold of finite volume, satisfying the following conditions for a positive constant $C_2$:\begin{itemize}
%\item[(i)] ${\underset{V \rightarrow 0} \liminf \ \frac{I_M(V)^{n+1}}{V^n}} \ \geq C_2>0$;
% \item[(ii)]$I^*<C_2.$
%\end{itemize}
%Then there exists a finite perimeter set $\Omega \subseteq M$ of positive volume such that 
%$I(\partial \Omega) \le \tilde{I}^*$.  %In particular, $\tilde{I}^*>0$. %Moreover, if $n+ 1 \leq 7$, then  $\Omega$ has smooth boundary.      \end{thm}

%The case of the plane gives equality for the functionals in Theorem \ref{ste5}. 

We are in need of some preliminary results, in order to justify this statement. The reader can find the following lemma in the special case of $n=1$ in \cite[Lemma 5.86]{chow}. The proof of our more general statement goes along the same lines as the one in \cite[Lemma 5.86]{chow}. We rewrite the detailed proof here for completeness's sake.

\begin{lem}\label{algebraicchow} 
For every positive real numbers $L_1,$ $L_2,$ $A_1$, $A_2$, $A_3$ and for every $n > 1$, we get
%\tiny
%\begin{equation}\label{Eq:algebraicchowStatement}
\[(L_1+L_2)^n\left(\frac{1}{A_1+A_2}+\frac{1}{A_3}\right)>\min\left\{L^n_1\left(\frac{1}{A_1}+\frac{1}{A_2+A_3}\right), L^n_2\left(\frac{1}{A_2}+\frac{1}{A_1+A_3}\right)\right\}.\]
%\end{equation}
%\normalsize
\end{lem}

\begin{proof}
By contradiction,  assume that $L_1, L_2, A_1, A_2, A_3 $ satisfy
\[ {(L_1+L_2)}^n\left(\frac{1}{A_1+A_2}+\frac{1}{A_3}\right)\leq L^n_1\left(\frac{1}{A_1}+\frac{1}{A_2+A_3}\right),\]
\[ {(L_1+L_2)}^n\left(\frac{1}{A_1+A_2}+\frac{1}{A_3}\right)\leq L^n_2\left(\frac{1}{A_2}+\frac{1}{A_1+A_3}\right).\]
We get the following inequalities
\[\frac{A_1(A_2+A_3)}{A_3(A_1+A_2)}\leq\frac{L^n_1}{{(L_1+L_2)}^n} \ \ \mathrm{and} \ \ \frac{A_2(A_1+A_3)}{A_3(A_1+A_2)}\leq\frac{L^n_2}{{(L_1+L_2)}^n}.\]
Summing up these  two last inequalities, we get
\[\frac{2A_1A_2}{A_3(A_1+A_2)}\leq\frac{L^n_1+L^n_2}{{(L_1+L_2)}^n}-1 = - \frac{1}{{(L_1+L_2)}^n} \ \sum^{n-1}_{k=1} \frac{n!}{(n-k)! k!}L^{n-k}_1 L^k_2  < 0.\]
This gives a contradiction.
\end{proof}
\begin{lem}\label{algebraicchow1-} 
For every positive real numbers $L_1,$ $L_2,$ $A_1$, $A_2$, $A_3$ and for every $n > 1$, we get
%\tiny
%\begin{equation}\label{Eq:algebraicchow1-Statement}
\[(L_1+L_2)^{n}\left(\frac{1}{A_1+A_2}+\frac{1}{A_3}\right)^n>\min\left\{L^{n}_1\left(\frac{1}{A_1}+\frac{1}{A_2+A_3}\right)^n, L^{n}_2\left(\frac{1}{A_2}+\frac{1}{A_1+A_3}\right)^n\right\}.\]
%\end{equation}
%\normalsize
\end{lem}
\begin{proof} Elevating to the power $n$ both sides of the equation in Lemma \ref{algebraic},  the result follows. \end{proof}
\begin{lem}[Second Condition of Superadditivity]\label{algebraicchow1} 
For every positive real numbers $L_1,$ $L_2,$ $A_1$, $A_2$, $A_3$ and for every $n > 1$, we get
%\tiny
%\begin{equation}\label{Eq:algebraicchow1Statement}
\[(L_1+L_2)^{n+1}\left(\frac{1}{A_1+A_2}+\frac{1}{A_3}\right)^n>\min\left\{L^{n+1}_1\left(\frac{1}{A_1}+\frac{1}{A_2+A_3}\right)^n, L^{n+1}_2\left(\frac{1}{A_2}+\frac{1}{A_1+A_3}\right)^n\right\}.\]
%\end{equation}
%\normalsize
\end{lem}

\begin{proof} Put $\alpha:=L_1+L_2$, $\beta:=\left(\frac{1}{A_1+A_2}+\frac{1}{A_3}\right)$, $\gamma:=\left(\frac{1}{A_1}+\frac{1}{A_2+A_3}\right)$, $\delta:=\left(\frac{1}{A_2}+\frac{1}{A_1+A_3}\right)$. With this notation in mind we have
  \begin{eqnarray*}
\alpha^{n+1}\beta^n & = & \alpha\alpha^{n}\beta^n\\
& > & \alpha\min\left\{L_1^n\gamma^n, L_2^n\delta^n\right\}\\
& = & \min\left\{\alpha L_1^n\gamma^n, \alpha L_2^n\delta^n\right\}\\
& > & \min\left\{L_1^{n+1}\gamma^n, L_2^{n+1}\delta^n\right\}.
\end{eqnarray*} 
So the result follows.
\end{proof}
%Incidentally, we find interesting the following remark even if we do not use it at all in this paper.
%\begin{rem} Looking at the notations in the proof of Lemma \ref{algebraicchow1}, we get  %\begin{equation}\label{Eq:algebraicchow1}
%\frac\beta\gamma+\frac\beta\delta-1=\frac{2A_1A_2}{A_3(A_1+A_2)}%\ge0.
%\end{equation}
%From which we deduce
%\begin{equation}\label{Eq:algebraicchow10}
%\beta=\frac{\gamma\delta}{\gamma+\delta}.
%\end{equation}
%It is, now immediate to conclude that
%\begin{equation}\label{Eq:algebraicchow11}
%\beta\le\gamma, \beta\le\delta. 
%\end{equation}
%\end{rem}
%Equation \eqref{Eq:algebraicchow11} give an account of the fact that a priori Lemmas \ref{algebraicchow}, \ref{algebraicchow1-}, \ref{algebraicchow1}, are not trivial and need something to be proved.

We may apply the same argument of Lemma \ref{connection}, in order to have topological information on the connected regions which appear in the proof of Theorem \ref{ste4}.

%\textcolor{red}{fare la versione $n$ dimensionale del seguente funzionale}.

\begin{cor}\label{connectionchow}Let $M$ be a complete Riemannian manifold of dimension $n+1$ of finite volume and let
 $\Omega$ be a finite perimeter set in  $M$ that minimizes $\tilde{I}^*$, i.e., 
\[J(\partial\Omega)=\mathcal{P}( \Omega)^{n+1}\left(\frac{1}{\mathrm{vol}(\Omega)}+\frac{1}{\mathrm{vol}(M)-\mathrm{vol}(\Omega)}\right)^n=\tilde{I}^*.\]
Then $\Omega$ and $M\setminus\Omega$ are connected, and $\tilde{I}^*=I=J$. 
%In particular,  if $n+1\le7$, then $\partial\Omega$ is a smooth hypersurface that separates $M$.
\end{cor}
%\begin{cor}\label{connectionchow}Let $M$ be a complete Riemannian manifold of dimension $ 2$ of finite volume and let
 %$\Omega$ be a finite perimeter set in  $M$ that minimizes $\tilde{I}^*$, i.e., 
%\[J(\partial\Omega)=\mathcal{P}( \Omega)^2  \left(\frac{1}{\mathrm{vol}(\Omega)}+\frac{1}{\mathrm{vol}(M)-\mathrm{vol}(\Omega)}\right)=\tilde{I}^*.\]
%Then $\Omega$ and $M-\Omega$ are connected,  $\partial \Omega$ separates  $M$ and $\tilde{I}^*=I=J$. In particular,  
 %$\partial\Omega$ is a smooth hypersurface that separates $M$.
%\end{cor}

\begin{proof} By Lemma \ref{algebraicchow1}, we  overlap the proof of Lemma \ref{connection} mutatis mutandis.
\end{proof}

\begin{thm} \label{ste61} If $M$ is a complete Riemannian manifold of finite volume, satisfying the following conditions for a positive constant $C_2$:
\begin{itemize}
\item[(i)] ${\underset{V \rightarrow 0} \liminf \ \frac{I_M(V)^{n+1}}{V^n}} \ \geq C_2>0$;
 \item[(ii)]$I^*<C_2,$
\end{itemize}
then there exists a connected isoperimetric region $\Omega \subseteq M$ of positive volume such that $M\setminus\Omega$ is connected, $0<\mathrm{vol}(\Omega)<\mathrm{vol}(M)$, and $I(\partial \Omega) = I^*>0$.  In particular, $\tilde{I}^*>0$ and,  if $n+ 1 \leq 7$, then $\partial\Omega$ is smooth, $\partial \Omega$  separates $M$  and 
 \begin{equation}\label{Eq:ste61Statement}
 I^*=\tilde{I}^*=J(\partial \Omega)=I(\partial \Omega) =J=I>0.
 \end{equation}
 Conversely, if  $M$ is noncompact and $\Omega \subseteq M$ is a connected unbounded finite perimeter set of positive volume such that $M\setminus\Omega$ is connected satisfying \eqref{Eq:ste61Statement}, then $\Omega$ is an isoperimetric region and both $(i)$ and $(ii)$ hold. 
 \end{thm}   

\begin{proof} We start proving the first implication. To this aim, note  that $I^* \le \tilde{I}^*$. By the definition of $I^*$, we may consider a minimizing sequence $\{\Omega_i\}_{i \in \mathbb{N}}$ such that $I(\partial \Omega_i)$ tends to $I^*$ for $i$ running to $\infty$. Now  it is easy to observe that $\mathrm{vol}(\Omega_i)+\mathcal{P}(\Omega_i)$ is uniformly bounded for all $i \in \mathbb{N}$. Putting $\mathrm{vol}(\Omega_i)=V_i$, the conditions (i) and (ii) together with Lemma \ref{equivcond} imply that $\{V_i\}_{i \in \mathbb{N}}$ does not tend neither to $0$ nor to $A=\mathrm{vol}(M)$. Hence there exists a $\delta >0$ such that $V_i \in \ [\delta, \mathrm{vol}(M) - \delta] $ for all $i \in \mathbb{N}$. From Corollary \ref{compattezza}, we may find a finite perimeter set $\Omega$ such that $\{\Omega_i\}_{i \in \mathbb{N}}$ converges to $\Omega$  in $L^1$--norm. Therefore 
\[\mathrm{vol}(\Omega)= \lim_{i \rightarrow \infty} V_i>0 ,\] 
and by the lower semicontinuity of the perimeters 
\[\mathcal{P}(\Omega) \le \liminf_{i \rightarrow \infty} \mathcal{P}(\Omega_i)\] 
we may deduce $I(\partial \Omega) = I^* \le \tilde{I}^*$ and moreover that $\Omega$ is an isoperimetric region and so if $n+ 1 \leq 7$, the standard regularity theory for isoperimetric regions recalled in Theorem \ref{regularity} ensures that $\partial\Omega$ is smooth. In principle by construction one can have  either $ I^*=\tilde{I}^*$ or  $I(\partial\Omega)<\tilde{I}^*$. In the first case the theorem follows. In the second case it remains to show that $\partial\Omega$ is separating, i.e., that $\Omega$ and $M\setminus\Omega$ are connected. This is indeed the case. We will show that $\partial\Omega$ is separating by contradiction.  Suppose that $\Omega$ is not connected. Without loss of generality we can assume that there is a partition $\Omega=\Omega_1\mathring{\cup}\Omega_2$ with $\Omega_1$ and $\Omega_2$ open sets with smooth boundary by Lemma \ref{algebraicchow1}. We have
$$
I^*=I(\partial\Omega)>\min\{I(\partial\Omega_1),I(\partial\Omega_2)\}\ge I^*,$$
 which is the desired contradiction. So $\Omega$ has to be connected. The same argument applied to $M\setminus\Omega$ in place of $\Omega$ shows that $M\setminus \Omega$ is connected too. Thus by definition $\partial\Omega$ is separating.
Then we have also that the converse $I(\partial \Omega) \ge \tilde{I}^*$ is true, because $\partial\Omega$ separating means that $\Omega$ belong to the family of admissible sets for the minimization problem associated to $\tilde{I}^*$. Hence $I(\partial \Omega) = \tilde{I}^*$ and the result follows.

Conversely, we prove the reverse implication. First of all,  note that the assumptions on $\Omega$ imply that it is an isoperimetric region. Now we argue by contradiction. Suppose that \[I^*\ge{\underset{V \rightarrow 0^+} \liminf \ \frac{I_M(V)^{n+1}}{V^n}},\] then \[0\le{\underset{V \rightarrow 0^+} \liminf \ I_M(V)^{n+1}}\le  \underset{V \rightarrow 0^+} \lim I^* V^n =I^* \underset{V \rightarrow 0^+} \lim V^n.\] It follows that \[\underset{V \rightarrow 0^+} \liminf \ I_M(V)^{n+1}\left[\frac{1}{V}+\frac{1}{\mathrm{vol}(M)-V}\right]^n=\underset{V \rightarrow 0^+} \liminf \ \frac{I_M(V)^{n+1}}{V^n}\le I^*.\] On the other hand, by definition we get $I^*\le\underset{V \rightarrow 0^+} \liminf \ I_M(V)^{n+1}\left[\frac{1}{V}+\frac{1}{\mathrm{vol}(M)-V}\right]^n$. Thus 
\begin{equation}\label{Eq:ste61}
I^*=\underset{V \rightarrow 0^+} \liminf \ I_M(V)^{n+1}\left[\frac{1}{V}+\frac{1}{\mathrm{vol}(M)-V}\right]^n.
\end{equation} 
Now we proceed exactly as in the proof of the second part of Theorem \ref{ste4}, in order to prove that 
\begin{equation}\label{Eq:ste610}
\lim_{V\to0^+}\frac{I_M(V)}{V}\le H(\partial\Omega)<+\infty.
\end{equation}
Notice that $H(\partial\Omega)<+\infty$ because $\Omega$ is an isoperimetric region of positive finite volume.
On the other hand by \eqref{Eq:ste61} we get that since $0<I^*<+\infty$ necessarily we have $ \lim_{V\to0^+}\frac{I_M(V)}{V}=+\infty$, in contradiction with \eqref{Eq:ste610}. The result follows.
\end{proof}
For the same reasons as in Theorem \ref{ste4} we see that we cannot avoid the assumption that $\Omega$ is unbounded.

\begin{rem}\label{Rem:HamiltonGeneralized} When we try to fulfill the hypothesis of Theorem \ref{ste61} applied to the case of $M$ being a compact manifold of dimension $n+1$, we see that $C_2\le(c_{n+1})^{n+1}$, where  $c_{n+1}:=(n+1)\omega_{n+1}^{\frac1{n+1}}$, is the Euclidean isoperimetric constant of dimension $n+1$.
\end{rem}

 In dimension $n+1\ge3$ the preceding theorems remain true if we replace in the statements $I_M$ by $\tilde{I}_M$, because of the following theorem.

\begin{prop}\label{Prop:EquivalenceInHigherDimensions1}
If $M$ is a Riemannian manifold of dimension  $n+1 \ge 3$ of finite volume, then $\tilde{I}_M=I_M$. In particular $\tilde{I}_M$ is continuous.
\end{prop}

\begin{proof} By  Definition \ref{hamilton}, $I_M\le \tilde I_M$.
Suppose that one has a minimizing sequence (i.e., $\mathcal{P}(\Omega_j)\to I_M(V)$) of connected $\Omega_j\subset M$ with volume $\mathrm{vol}(\Omega_j)=V>0$ for every $j$, in dimension  $n + 1 \ge2$ and assume that $M\setminus\Omega_j$ is not connected than one can consider the construction schematically represented in Fig. 4 below.

Take a curve $\gamma_j:[0,1]\to M$, $\gamma_j(]0,1[)\subset M\setminus\Omega_j$ joining two distinct points in the boundary of $\gamma_j(0),\gamma_j(1)\in\partial\Omega_j$, $\gamma_j(0)\ne\gamma_j(1)$, then take a small tubular neighborhood of $\gamma_j$ of small area $A_j\to0$ and small volume $\Delta v_j\to0$ to be chosen in such a way $\tilde c_j\Delta v_j\to0$, where $\tilde{c}_j=\tilde c_j(\Omega_j)$ is an upper bound on the modulus of the mean curvature of $\partial\Omega_j$, and glue this thin tube smoothly to $\partial\Omega_j$ then compensate the volume added or subtracted in the gluing procedure in a far point of $\partial\Omega_j$,  and one obtain a connected competitor $\Omega'_j$, such that $M\setminus\Omega_j$ is connected and $\mathrm{vol}(\Omega'_j)=V$ and 
$$|\mathcal{P}(\Omega_j)-\mathcal{P}(\Omega'_j)|=\Delta A_j+\Delta A'_j\le \Delta A_j+\tilde c_j(\Omega_j)\Delta v_j\to0,$$
where $\Delta A'_j\le \tilde{c}_j(\Omega_j)\Delta v_j$.

\begin{figure}[h!]\label{fig:1}
 \includegraphics[scale=1]{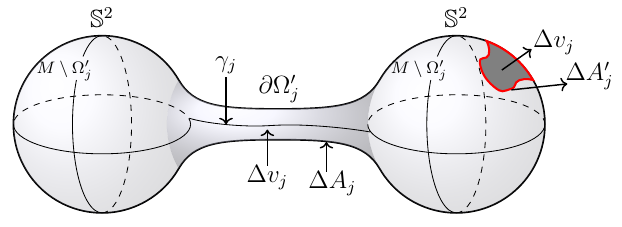}
    \caption{An example in dimension $n+1\ge3$ of the competitor $\Omega'_j$ constructed from $\Omega_j$ in the proof of Proposition \ref{Prop:EquivalenceInHigherDimensions1}.}
\end{figure}

So $I_M(V)=\lim_{j\to+\infty}\mathcal{P}(\Omega'_j)\ge\tilde{I}_M(V)$. 

\end{proof}

%In order to show an important application of the above results, we must recall some well known facts which are known from geometric analysis and geometric measure theory.  First of all the notion of \textit{rectifiable varifold} $\underline{v}\left(M_{j}, \theta_{j}\right)$ of dimension $k$ in a Riemannian manifold can be found in \cite{simon}. That of generalized mean curvature vector, of weight of a varifold and that of density can be found in \cite{simon} as well.

The following remark is appropriate here.

\begin{rem} Note that in dimension $2$ the two minimization problems whose infima are respectively $\tilde{I}^*$ and $I^*$ are not equivalent. This does not happen for dimensions $\ge 3$. See Remark \ref{Rem:ComparisonWithHsu}.
\end{rem}

Now we may improve Theorem \ref{ste61} in the case of dimension greater or equal than two, considering an ananologous statement for $\tilde{I}_M$ instead of $I_M$.

\begin{thm}\label{ste5bis}
If  $M$ is  a Riemannian manifold of finite volume of dimension $n+1 \ge 2$ satisfying  
\begin{enumerate}
\item[(i).] \label{Eq:ste5bis}$\underset{V \rightarrow 0} \liminf \ \frac{{\tilde{I}_M(V)}^{n+1}}{V^n} \ \geq C_3>0;$
 \item[(ii).]\label{Eq:ste5bis0} $\tilde{I}^*< C_3;$
\end{enumerate}
then $\tilde{I}^*>0$. Moreover, if $n+1 \ge 3$, then $\tilde{I}^*=I^*>0$ and the equality is achieved by a connected isoperimetric region $\Omega$ with connected complementary set $M\setminus\Omega$. Furthermore, if $n+1\le7$, then $\partial\Omega$ is smooth and separating. 

Conversely, if  $n+1\ge3$, $M$ is noncompact and $\Omega \subseteq M$ is a connected unbounded finite perimeter set of positive volume such that $M\setminus\Omega$ is connected satisfying \eqref{Eq:ste61Statement}, then $\Omega$ is an isoperimetric region and both $(i)$ and $(ii)$ hold. %Conversely, if $n+1\ge3$ and the infimum $\tilde{I}^*>0$ is achieved then  and $(i)$, $(ii)$ hold.%$($because of the regularity properties of isoperimetric regions in low dimension$)$.
\end{thm}
\begin{proof} Consider a finite perimeter  minimizing sequence $\{\Omega_i\}_{i \in \mathbb{N}}$ for $\tilde I^*$, namely $I(\partial\Omega_i)\to\tilde{I}^*$ with each $\partial \Omega_i$  separating.  Corollary \ref{compattezza} implies that there exists a finite perimeter set $\Omega$ such that a sub-sequence $\Omega_i$ converges to $\Omega$ in $L^1(M)$ and $0\le I(\partial\Omega)\le\tilde{I}^*$. On the other hand, Corollary \ref{connectionchow} implies that $\partial \Omega$ separates $M$, if the infimum is achieved by $\Omega$.

\medskip
\medskip

Case (j). $I(\partial\Omega)=0$ and $\tilde{I}^*=I(\partial\Omega)$. Now, observe that $I(\partial\Omega)$ cannot be zero because this means that $\mathcal{P}(\Omega)=0$, hence either  $\mathrm{vol}(\Omega)=0$ or $\mathrm{vol}(\Omega)=\mathrm{vol}(M)$. If $\mathrm{vol}(\Omega)=0$, then  $I(\partial\Omega) \neq 0$, because otherwise it would be $\mathrm{vol}(\Omega_i)\to 0$ and the assumption would imply that \[\tilde I^*=\lim_{i\to\infty}I(\partial\Omega_i)\ge C_3>0,\] which is impossible since by assumption it holds (ii). The same argument applies to $M\setminus\Omega$, if one assumes that $\mathrm{vol}(M)=\mathrm{vol}(\Omega)$. This is enough to ensure that $\tilde{I}^*>0$ in any dimension but it is not enough to ensure that $\tilde I^*$ is achieved.

\medskip
\medskip

Case (jj). $I(\partial\Omega)>0$ and $\tilde{I}^*=I(\partial\Omega)$. From Lemma \ref{algebraicchow1}, we note that there is superaddivity of the functional $I$ and so $\partial\Omega$ is separating (in any dimension). The result follows in this case.

\medskip
\medskip

Case (jjj). $I(\partial\Omega)=0$ and $\tilde{I}^*>I(\partial\Omega)$. As in case (j) we observe that $I(\partial\Omega)$ cannot be zero because this means that $\mathcal{P}(\Omega)=0$, hence either  $\mathrm{vol}(\Omega)=0$ or $\mathrm{vol}(\Omega)=\mathrm{vol}(M)$. If $\mathrm{vol}(\Omega)=0$, then  $I(\partial\Omega) \neq 0$, because otherwise it would be $\mathrm{vol}(\Omega_i)\to 0$ and it would imply that 
\[\tilde I^*=\lim_{i\to\infty}I(\partial\Omega_i)\ge\underset{V \rightarrow 0} \liminf \ \frac{{\tilde{I}_M(V)}^{n+1}}{V^n}\ge C_4>0,\] which is impossible since by assumption (jj) we have $\tilde{I}^*< C_4$. The same argument applies to $M\setminus\Omega$, if one assumes that $\mathrm{vol}(M)=\mathrm{vol}(\Omega)$. Observe that at this point by (j)-(jj)-(jjj) we already proved that $\tilde{I}^*>0$.

At this point the only case that it remains to treat is the following.

\medskip
\medskip

Case (jv). $0 < I (\partial \Omega) < \tilde{I}^*$.

In dimension $2$ there is nothing to do, unfortunately Case (jv) could happen as showed in Example \ref{newexample}. So we cannot prove the existence of a minimizer for $\tilde{I}$ but just that the infimum $\tilde{I}^*$ is positive, but not necessarily achieved. On the other hand when $n+1\ge3$ we can prove also that the the infimum $\tilde{I}^*$ is achieved. So in any dimension $\tilde{I}^*$ is positive, howevere it is achieved only if $n+1\ge3$.

\medskip
\medskip

%Case (iv.1). We will treat before the case $n+1\ge3$. 

In dimension greater or equal to $\ge3$, Propositions \ref{Prop:EquivalenceInHigherDimensions0}  and \ref{Prop:EquivalenceInHigherDimensions1} imply  $I^*=\tilde{I}^*$ and a fortiori  $I^*>0$. This means that we can now look at minimizing sequences $\Omega'_j$ for $\tilde I^*$ that are minimizing even  for $I^*$. Then $\Omega'_j$ could be chosen to be a sequence of isoperimetric regions, in order to deduce that the limit $\Omega'$ is an isoperimetric region with volume $0<\mathrm{vol}(\Omega')<\mathrm{vol}(M)$. Now $I(\partial\Omega')=I^*=\tilde{I}^*$. By the same argument of Lemma \ref{algebraicchow1} we show that $\Omega'$ is separating because this time $I(\partial\Omega')=I^*$. To prove the converse part of the statement of the theorem we proceed exactely as in the proof of Theorem \ref{ste61}.
\end{proof}

A similar argument applies to the following result.

\begin{thm} \label{ste4bis1} If  $M$ is  a Riemannian manifold of finite volume of dimension $n+1\ge 2$ satisfying  
\begin{itemize}\label{Eq:ste5bis01} 
\item[(i).]\(\underset{V \rightarrow 0} \liminf \ \frac{{\tilde{I}_M(V)}}{V} \ \geq C_4>0\);
\item[(ii).]\(\tilde{I}^\flat<C_4;\)
\end{itemize}
then $\tilde{I}^\flat>0$. In particular if $n+1\ge 3$, then $\tilde{I}^\sharp=I^\flat>0$ and is achieved by a connected isoperimetric region $\Omega$ of positive volume with connected complementary set $M\setminus\Omega$. Furthermore, if $n+1\le7$, then $\partial\Omega$ is smooth and separating.

Conversely, if  $n+1\ge3$, $M$ is noncompact and $\Omega \subseteq M$ is a connected unbounded finite perimeter set of positive volume such that $M\setminus\Omega$ is connected satisfying \eqref{Eq:Ste4Statement}, then $\Omega$ is an isoperimetric region and both $(i)$ and $(ii)$ hold.   
\end{thm}

\begin{proof}We may argue as in  the proof of Theorem \ref{ste5bis}, but instead of Corollary \ref{connectionchow} we use Lemma \ref{connection}. To prove the converse part of the statement of the theorem we proceed mutatis mutandis as in the proof of Theorem \ref{ste4}.
\end{proof}

It remains to show that $\partial\Omega$, which  has been constructed as in the previous  proof, turns out to be separating even in dimension two. We are unable for the moment to prove such a statement in  full generality considering the problem of minimizing over domains whose boundary have more than one connected component, because   the argument of Proposition \ref{Prop:EquivalenceInHigherDimensions0} does not work anymore in dimension two. However we have in the next theorem a condition that generalises a previous result obtained by Hamilton in \cite{hamilton4} for compact manifolds. Roughly speaking, if we change a little bit the minimization problem in dimension two looking to a smaller class of admissible sets, we then can prove the existence of a minimizer. The detailed technical statement is the theorem below.    

\begin{thm}\label{ste5bis11}Let  $M$ be  a Riemannian manifold of finite volume of dimension two and define the following restriction of $\tilde{I}_M(V)$\[\hat{I}_M(V):=\inf \{\mathcal{P}(\Omega) \ | \ \Omega \subseteq M, \  \mathrm{vol}(\Omega)=V,\partial \Omega=F(\mathbb{S}^1)\}, \ \]where the infimum is taken over the family of all Lipschitz simple loops \(F:\mathbb{S}^1\to M \)  with \(\partial \Omega \) separating \( M\) and \[\hat{I}^*_M(V)={(\hat{I}_M(V))}^2 \cdot {\left(\frac{1}{V}+\frac{1}{\mathrm{vol}(M)-V}\right)} ,\]\[\hat{I}^*= \inf \{\hat{I}^*_M(V) \ | \ V \  \in \ ]0, \mathrm{vol}(M)[ \  \}. \]  If the following conditions are satisfied  \begin{itemize}\item[(i).]\label{Eq:ste5bis1}$\underset{V \rightarrow 0} \liminf \ \frac{{\hat{I}_M(V)}^2}{V} \ \geq C_5>0;$ \item[(ii).] \label{Eq:ste5bis10} $\hat{I}^*< C_5;$
\item[(iii).] $C_5\le(c_2)^2=4\pi>0$, where $c_2$ is the Euclidean isoperimetric constant of dimension $2$.
\end{itemize}
then,  there exists a connected isoperimetric region $\Omega \subseteq M$ of positive volume with connected complementary set $M\setminus\Omega$ such that $I(\partial \Omega) = \hat{I}^*>0$ and $\Omega$ has  smooth embedded boundary, which separates $M$, has just one connected component and is a simple loop.

Conversely, if  $n+1=2$, $M$ is noncompact and $\Omega \subseteq M$ is a connected unbounded finite perimeter set of positive volume such that $M\setminus\Omega$ is connected satisfying $I(\partial \Omega) = \hat{I}^*>0$, then $\Omega$ is an isoperimetric region and both $(i)$ and $(ii)$ hold.
\end{thm}

\begin{proof} The proof goes exactly as Cases (j), (jj), (jjj) of the proof of Theorem \ref{ste5bis}.   By assumption (i),  we cannot have that for every compact set $K\subset\subset M$, there is an index $i_K>0$ such that $K\cap\partial\Omega_i=\emptyset$ for every $i\ge i_K$ because $\mathrm{vol}(M\setminus\Omega_i)$ is bounded from below away from zero. This means that (up to a subsequence renamed again $\partial\Omega_i$) there exists a sequence of points $x_i\in\partial\Omega_i$ and a compact set $K\subset\subset M$ such that $x_i\in K$ for every $i\in\mathbb{N}$.  Now, by definition $\partial\Omega_i$ is connected and we are considering a minimizing sequence, therefore the perimeter is bounded, that is,  the length $\mathcal{P}(\Omega_i)\le C^*$ is uniformly bounded by a positive constant $C^*>0$, $x_i\in\partial\Omega_i$  and so \[\partial\Omega_i\subseteq U_{C^*}(K):=\{x\in M \ : \ d(x,K)\le C^*\}\subset\subset K_1\subset\subset M,\] for every $i\in\mathbb{N}$ (observe that this fact is specific of the dimension $2$, and $d$ is the Riemannian metric on $M$). Since $U_{C^*}(K)$ is compact by the Hopf-Rinow-Heine-Borel Theorem (see \cite{chavel}) in complete Riemannian manifolds, we reduced the problem to the compact case. Now we can apply theorem Theorem \ref{ste61} with $C_2$ being any positive constant such that $C_2\le(c_2)^2=4\pi$ (compare Remark \ref{Rem:HamiltonGeneralized}), because of our assumption $(iii)$ and the fact that our compact $K_1$ contain a minimizing sequence of the original problem ensuring that $\hat{I}^*_{K_1}=\hat{I}^*_M$. Hence $\partial\Omega$ is smooth and embedded. A priori $\partial\Omega$ could have more than one connected component and not being separating, but this is not actually the case because by Ascoli-Arzelà's Theorem the isoperimetric region $\Omega$ and $\partial\Omega$ are the Hausdorff limit of a sequence of separating $\Omega_i$ with one connected component $\partial\Omega_i$, that is, $d_\mathrm{H}(\Omega_i,\Omega)\to0$, $d_\mathrm{H}(\partial\Omega_i,\partial\Omega)\to0$  when $i \to \infty$ and so also $d_\mathrm{H}(M\setminus\Omega_i,M\setminus\Omega)\to0$. 
\begin{figure}[h!]\label{Fig7}
 \includegraphics[scale=0.7]{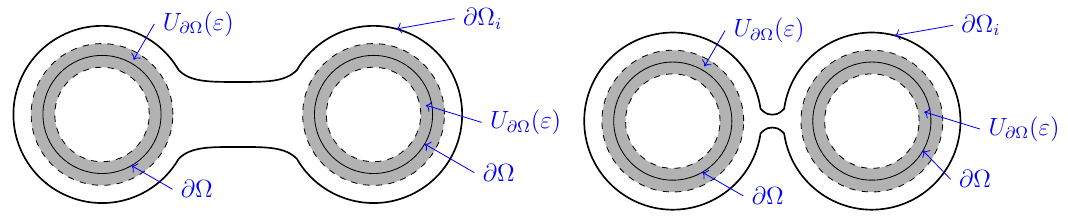}
    \caption{Two connected components of $\partial\Omega$. $U_{\partial\Omega}(\varepsilon):=\{x\in M: d(x,\partial\Omega)<\varepsilon\}$ there is still no loss of perimeter in the sequence $\Omega_i$.}
\end{figure}
%\begin{figure}[h!]\label{Fig8}
 %\includegraphics[scale=0.8]{Nar2020_10C.pdf}
    %\caption{Two connected components of $\partial\Omega$: there is a loss of perimeter in the limit.}
%\end{figure}
\begin{figure}[h!]\label{Fig9}
 \includegraphics[scale=0.8]{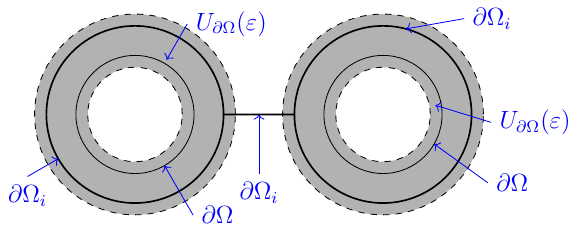}
    \caption{Two connected components of $\partial\Omega$, $\partial\Omega_i$ is not able to stay in a small tubular neighborhood of $\partial\Omega$ maintaining the property of being connected. Moreover, there is a loss of perimeter in the limit.}
\end{figure}
We omit the easy details needed to prove this last assertion. To prove the converse part of the statement of the theorem we proceed mutatis mutandis as in the proof of Theorem \ref{ste61}.

\end{proof}

Now we discuss the remaining functionals in dimension two.
\begin{thm}\label{ste5bis11Hsu}
Let  $M$ be  a Riemannian manifold of finite volume of dimension two and define the following restriction of $\tilde{I}_M(V)$\[\hat{I}_M(V):=\inf \{\mathcal{P}(\Omega) \ | \ \Omega \subseteq M, \  \mathrm{vol}(\Omega)=V,\partial \Omega=F(\mathbb{S}^1)\}, \ \]where the infimum is taken over the family of all Lipschitz simple loops \(F:\mathbb{S}^1\to M \)  with \(\partial \Omega \) separating \( M\),and
\[\hat{I}^\flat_M(V)=\hat{I}_M(V) \cdot \left(\frac{1}{V}+\frac{1}{\mathrm{vol}(M)-V}\right) ,\ \mbox{and} \ \hat{I}^\flat= \inf \{\hat{I}^\flat_M(V) \ | \ V \  \in \ ]0, \mathrm{vol}(M)[ \  \}. \]  Assume the following conditions are satisfied
\begin{itemize}
\item[(i).]$\underset{V \rightarrow 0} \liminf \ \frac{{\hat{I}_M(V)}}{V} \ \geq C_5>0;$
 \item[(ii).]  $\hat{I}^\flat< C_5;$
\end{itemize}
then, there exists a connected isoperimetric region $\Omega \subseteq M$ of positive volume with connected complementary set $M\setminus\Omega$ such that $C(\partial \Omega) = \hat{I}^\flat$ and $\Omega$ has smooth boundary, which separates $M$ and just one connected component and is a simple loop.

Conversely, if  $n+1=2$, $M$ is noncompact and $\Omega \subseteq M$ is a connected unbounded finite perimeter set of positive volume such that $M\setminus\Omega$ is connected satisfying $C(\partial \Omega) = \hat{I}^\flat$, then $\Omega$ is an isoperimetric region and both $(i)$ and $(ii)$ hold.
\end{thm}

\begin{proof} The proof goes exactly as Cases (j), (jj), (jjj) of the proof of Theorem \ref{ste5bis}.  Here we prove just what it is needed to complete the proof. By assumption (i),  we cannot have that for every compact set $K\subset\subset M$, there is an index $i_K>0$ such that $K\cap\partial\Omega_i=\emptyset$ for every $i\ge i_K$ because $\mathrm{vol}(M\setminus\Omega_i)$ is bounded below away from zero. This means that (up to a subsequence renamed again $\partial\Omega_i$) there exists a sequence of points $x_i\in\partial\Omega_i$ and a compact set $K\subset\subset M$ such that $x_i\in K$ for every $i\in\mathbb{N}$.  Now, by definition $\partial\Omega_i$ is connected and we are considering a minimizing sequence, therefore the perimeter is bounded, that is,  the length $\mathcal{P}(\Omega_i)\le C^*$ is uniformly bounded by a positive constant $C^*>0$, $x_i\in\partial\Omega_i$  and so \[\partial\Omega_i\subseteq U_{C^*}(K):=\{x\in M \ : \ d(x,K)\le C^*\}\subset\subset K_1\subset\subset M\] for every $i\in\mathbb{N}$ (observe that this fact is specific of the dimension $2$, and $d$ is the Riemannian metric on $M$). Since $U_{C^*}(K)$ is compact by the Hopf-Rinow-Heine-Borel Theorem (see \cite{chavel}) in complete Riemannian manifolds, we reduced the problem to the compact case. This case is indeed a special case of Theorem \ref{ste61}, in which $C_2$ could be any positive constant. Hence $\partial\Omega$ is smooth and embedded. A priori $\partial\Omega$ could have more than one connected component and not being separating, but this is not actually the case because by Ascoli-Arzelà's Theorem $\Omega$ and $\partial\Omega$ are the Hausdorff limit of a sequence of separating $\Omega_i$ with one connected component $\partial\Omega_i$, that is, $d_\mathrm{H}(\Omega_i,\Omega)\to0$, $d_\mathrm{H}(\partial\Omega_i,\partial\Omega)\to0$  when $i \to \infty$. We omit the easy details needed to prove this last assertion.
Thus easily the first part of theorem follows. To prove the converse part of the statement of the theorem we proceed mutatis mutandis as in the proof of Theorem \ref{ste4}.
\end{proof}

%\begin{rem} The by now classical variation argument of \cite{hamilton4} applies and we can compute the mean curvature of $\partial\Omega$, denoted by $H_{\partial\Omega}$ where $\Omega$ is the minimizer found in Theorem \ref{ste5bis11Hsu}. So we get
% \begin{equation}\label{Eq:MeanCurvatureOfMinimizers}
% H_{\partial\Omega}=\mathcal{P}(\Omega)\left[\frac1{\mathrm{vol}(\Omega)}-\frac1{\mathrm{vol}(M)-\mathrm{vol}(\Omega)}\right].
% \end{equation}
%\end{rem}

%\begin{proof}We may argue as in  the proof of Theorem \ref{ste5bis11}, but instead of Theorem \ref{ste61} we use Theorem \ref{ste4}.
%\end{proof}

%Theorem \ref{ste5bis11}  results of Hamilton in \cite{hamilton4}. 

Compare now Theorems \ref{ste5bis}, \ref{ste4bis1},   and \ref{ste5bis11Hsu} with Remark \ref{Rem:ComparisonWithHsu}, in order to understand the degree of generalizations of the above results. For instance,  from Remark \ref{Rem:ComparisonWithHsu}, it is easy to check that under the assumptions of  Theorem \ref{hsu1}, the hypothesis of our Theorem \ref{ste4bis1}  are satisfied. Thus all the noncompact two dimensional manifolds treated in \cite{hamilton1}, \cite{hamilton2}, \cite{hsu}, \cite{delpino1}, \cite{delpino2}, fall under the scope of Theorem \ref{ste5bis11Hsu}.

 %it is showed thaBy  Lemma $2.1$ of \cite{hsu} the can reduce to the compact setting and there we can apply Berard-Meyer. Under the assumptions of Theorem $1.1$ of \cite{hsu} we observe that assumptions of our Theorem \ref{ste5bis11Hsu}, are satisfied with $C_7=+\infty$, because $\frac{{\hat{I}_M(V)}}{V}\ge \frac{I_M(V)}{V}\sim\frac{c_2\sqrt{V}}{V}\to+\infty$, as $V\to0^+$. 

Most of the ideas and of the arguments that we have seen in the present paper apply not only to the functionals studied by Hamilton and Daskalopulos  \cite{hamilton1, hamilton2, hamilton3, hamilton4} but might work for much more general functionals which satisfy prescribed conditions of superadditivity. Indeed we conjecture that:
%\begin{rem} Observe that the proof of the preceding theorem combined with Lemma $2.1$ of \cite{hsu} permit to obtain an alternative proof of Theorem $1.1$ of \cite{hsu}. \end{rem}

\begin{conj} If we replace the functions \[(A,V)\mapsto A\left[\frac{1}{V}+\frac{1}{\mathrm{vol}(M)-V}\right],\ \mbox{and} \ (A,V)\mapsto A^{n+1}\left[\frac{1}{V}+\frac{1}{\mathrm{vol}(M)-V}\right]^n,\] with more general functions which are lower semicontinuous in the variable $A$, continuous in the variable $V$, and  satisfy a superadditivity condition as in Lemmas \ref{algebraic} and \ref{algebraicchow1}, then  statements analogous to that of Theorems \ref{ste5bis}, \ref{ste4bis1},  and \ref{ste5bis11Hsu} hold.
\end{conj}

%Since we do not see any application of such a general statement we do not prove it here and leave it as a conjecture and/or open problem.

\section{Some examples}

The difficulty of applying the argument of Theorem \ref{ste4} is due to the fact that  $\tilde{I}_M(V)$ may be or not continuous. 
 In the present section, we provide some examples in order to show different behaviours, when we test the condition (i) of Theorem \ref{ste4} on  complete Riemannian manifolds of finite volume. Of course, these behaviours depends on the metric which we are considering.
\begin{ex} The manifolds considered in \cite{hamilton1}, \cite{hamilton2}, \cite{hamilton3}, \cite{hamilton4}, \cite{hsu}.
\end{ex}
\begin{ex}\label{optimal}
The present example illustrates Theorem \ref{ste4}.  We take a rotationally symmetric surface $M \simeq \mathbb{R}^2$, that is, diffeomorphic to the usual plane (see \cite{ritore2}). Fix the origin $p \in M$ and a smooth metric $g$ which can be written in normal polar coordinates by $g=dt^2+{f(t)}^2d\theta^2$, where
$(t,\theta) \in \ ]0,\infty[ \times S^1$ and $S^1=\{(x,y)\in \mathbb{R}^2 \ | \ x^2+y^2=1\}$ is endowed with the Riemannian metric induced by that of the Euclidean plane. Here $d\theta$ is the Riemannian volume form of $S^1$ and $f : t \in ]0,\infty[ \mapsto f(t) \in ]0,\infty[$ is such that at the origin we can extend $g$ to a smooth metric on the entire plane.
This construction it is always possible for suitable $f$, see \cite[p.13]{petersen}. Now $M$ has  sectional curvature  \[K(t)=-\frac{f''(t)}{f(t)},\] 
(see \cite{petersen} for the details). Note that the length $l(c_t)$ of a geodesic circle $c_t$ depends on $f(t)$, namely $L(c_t)=2\pi  f(t)$. We require the following restrictions on $f(t)$: 

\medskip
\medskip

(j).   There is $t_1>0$ such that $f'(t)>0$ for $t<t_1$ and $f'(t)<0$ for $t>t_1$;

\medskip
\medskip

(jj).  $K'(t)\le 0$ for $t$ large enough;

\medskip
\medskip

(jjj). $\mathrm{vol}(M^2)=\int^\infty_0 2\pi f(r) dr < \infty$;

\medskip
\medskip

(jv). $V(t_0)=\int^\infty_{t_0} 2\pi f(r) dr < \infty$ is equal to $\mathrm{vol}(\Omega(t_0))$, where $t_0>0$ is fixed and $\Omega(t_0)=\{(t, \theta) \in \ ]0,\infty[ \times S^1 \ | \ t>t_0 \}$.

\medskip
\medskip

With the assumptions (j)--(jv), we obtain a plane with decreasing curvature as in \cite[Lemma 2.1, pp. 1104--1105]{ritore2} and apply \cite[Lemma 3.1]{morgan4}, in order to find that $I_M(V(t))=l(c_t)$. Historically, we mention that the above relation $I_M(V(t))=l(c_t)$ is found also in \cite{cao, pansu, topping}. Then  
\[\lim_{V(t) \rightarrow 0}   \frac{I_M(V(t))}{V(t)}=\lim_{t \rightarrow \infty}   \frac{L(c_t)}{V(t)}=\lim_{t \rightarrow \infty}   \frac{f(t)}{\int^\infty_t  f(r) dr}=\lim_{t \rightarrow \infty}   \frac{-f'(t)}{f(t)}.\]
Now if \[f(t)= \left\{\begin{array}{lcl} e^{-t},&\,\,& t\ge t_0,\\
h(t),&\,\,& t_1<t \le t_0,\\
\sin(t), &\,\,& t\in[0,t_1],\end{array}\right.\]
where $h : t \in [t_1,t_0] \mapsto h(t) \in \ ]0,\infty[ $ is a smooth function with $h'(t)<0$ (by Sturm-Liouville theory) , having $-\frac{h''}{h}$ nonincreasing, and such that the globally defined function $g$ is smooth, then (j) is satisfied, (jj) becomes $K(t)=-1$ for all $t >t_0$ and the integrals in (jjj) and (jv) are always well defined. Here the above limit becomes
\begin{equation}\label{Eq:Example5.2}
\lim_{V(t) \rightarrow 0}   \frac{I_M(V(t))}{V(t)}=\lim_{t \rightarrow \infty}   \frac{-g'(t)}{g(t)}=\lim_{t \rightarrow \infty}   -\frac{-e^{-t}}{e^{-t}}=1=C_1.
\end{equation}
Therefore the condition (i) of Theorem \ref{ste4} is true. It remains to check that also the condition (ii) of Theorem \ref{ste4} is true. Here we make suitable choices for $t_0$, $t_1$, and $h$ to be sure that there exists a $t\leq t_1$ such that 
\[I^\flat_M(\mathrm{vol}(B(p,t)))=L(c_t) \ \left(\frac{1}{\mathrm{vol}(B(p,t))}   + \frac{1}{ \mathrm{vol}(M)-\mathrm{vol}(B(p,t))} \right)<1,\]
for a geodesic ball $B(0,t)$ in $M$, centered at the origin and of radius $t$. With this aim in mind fix $t_0\geq\max\{\pi+3, -\log(\frac{\alpha_1}{2})\}$ where $\alpha_1=\sin(\bar{t})$ and $\frac{\pi}{2}<\bar{t}<\pi$ is such that $\frac{\sin(\bar{t})}{1-\cos(\bar{t})}=\frac{1}{2}$. Hence if $\bar{t}\leq t\leq\pi$, then $\frac{\sin(t)}{1-\cos(t)}<\frac{1}{2}$. Choose $t_1$ such that $\bar{t}\leq t_1<\pi$ and $1<\sin(t_1)e^{t_0}<2$. Writing the preceding expression for our concrete example for some $\frac{\pi}{2}\leq\bar{t}\leq t\leq t_1$ , yields
$$I^\flat_M(\mathrm{vol}(B(p,t)))\leq\sin(t)\left[\frac{1}{1-\cos(t)}+\frac{1}{e^{-t_0}+\alpha}\right],$$
where $\alpha:=\int_t^{t_0}g(t)dt$. But $\alpha\geq e^{-t_0}(t-t_0)$, because $h$ is chosen nonincreasing. Thus
\begin{eqnarray}
I^\flat_M(\mathrm{vol}(B(p,t))) & \leq & \sin(t)\left[\frac{1}{1-\cos(t)}+\frac{1}{e^{-t_0}(t_0-\pi+1)}\right]\\
& < & \frac{1}{2}+\frac{\sin(t)}{4e^{-t_0}}<\frac{1}{2}+\frac{1}{2}=1.
\end{eqnarray}

Now we invoke \cite[Lemma 3.1, B]{morgan4} because $M$ has nonincreasing sectional curvature and finite volume and note that the isoperimetric profile $I_M(V)$ is minimized over the balls of $M$ or over the complement of a ball. This allows us to conclude that
\[I^\flat=\inf\{I^\flat_M(V) \ | \ V \in \ ]0,\mathrm{vol}(M)[\}<1.\]
\end{ex}

Example \ref{optimal} may be modified in various ways.

\begin{ex}\label{optimalbis} 

In Example \ref{optimal},  we may consider a different function $g(t)$ satisfiying the conditions (j)--(jv), but not of exponential type for large values of $t$. Of course, the constant $C_1$ will be different and the limit $\lim_{V \rightarrow 0} I_M(V)/V$ will require a slight different solution, but it exists and we argue in the same way,   finding further families of examples for Theorem \ref{ste4}. 
\end{ex}

Now we discuss the sufficient and necessary conditions of Theorem \ref{ste4} via an example.

\begin{ex}\label{newexample} 

 Since we will deal with different metrics $g$ on different manifolds $M$, we stress this dependence of $I^\flat$ on $M$ and $g$ via the notation 
\begin{equation}\label{BF1}
I^\flat_{M,g}=\inf\left\{I_{M,g}(V) \cdot \left(\frac{1}{V}+\frac{1}{\mathrm{vol}_g(M)-V}\right) | \ V \in \ ]0,\mathrm{vol}_g(M)[ \right\}.\end{equation} Consider the two dimensional canonical sphere $S^2$ with the round metric $\rho$ and  $I^\flat_{S^2,\rho}$.

 We are going to construct a two dimensional  Riemannian manifold $(M,g_{t_0, l, t_1})$ with metric $g_{t_0, l, t_1}$ depending on three parameters only, introducing first the parameter $t_0$, then $l$ and finally $t_1>t_0$. This $(M,g_{t_0, l, t_1})$ turns out to have volume  $4\pi$ and nonincreasing sectional curvature $K(t)$ such that $-1 \le K(t) \le 1$ with $K(t)=1$ on $[0,t_0]$ and  $K(t)=-1$ on $[t_1, + \infty]$. Such an $(M,g_{t_0,t_1, l})$ will satisfy the necessary conditions of Theorem \ref{ste4}, but not the assumptions (i) and (ii) with $I^\flat=I^\flat_{S^2,\rho}$. 
\begin{figure}[H]\label{Fig10}
 \includegraphics[scale=1]{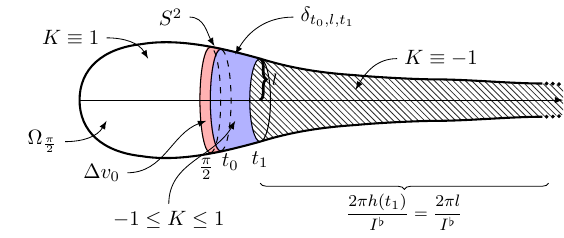}
    \caption{An example of a manifold $(M,g_{t_0, l, t_1})$ having volume  $4\pi$, nonincreasing sectional curvature $K(t)$ such that $-1 \le K(t) \le 1$ with $K(t)=1$ on $[0,t_0]$ and  $K(t)=-1$ on $[t_1, + \infty]$, that satisfy the necessary conditions of Theorem \ref{ste4}, but not the assumptions $(i)$ and $(ii)$ with $I^\flat=I^\flat_{S^2,\rho}$.}
\end{figure}

We begin to consider $\pi>t_0>\frac\pi2$ and $0<l<\sin(t_0)<1$ such that \begin{equation}\label{BF2}
t_0+\frac1l-\frac1{I^\flat}+\cos(t_0)<t_1<t_0+\frac1l<\pi.\end{equation}  The parameter $t_0$ will be chosen later. Observe that $I^\flat_{S^2,\rho}=1$. By an explicit calculation it is possible to show that $I^\flat_{S^2,\rho}$ is attained when $\Omega_t=\Omega_{\frac\pi2}$, i.e., when is equal to the half of the sphere. %Observe that $I_r^\flat=\frac1rI^\flat_1$, so we choose $r$ large enough, e.g., $r>I^\flat_1>0$ to have $I^\flat_r<1$. 
We define $f : t\in [0,+\infty[ \to f(t) \in [0,+\infty[$ as follows 
\begin{equation}\label{BF3}
f(t)= \left\{\begin{array}{lcl} le^{-I^\flat(t-t_1)},&\,\,& t\ge t_1,\\
\\
h_{t_0,l,t_1}(t),&\,\,& t_0<t \le t_1,\\
\\
\sin(t), &\,\,& t\in[0,t_0],\end{array}\right.
\end{equation}
where $h_{t_0,l,t_1} : t \in [t_0,t_1] \mapsto h_{t_0,l,t_1}(t) \in \ ]0,\infty[ $ is a smooth function with $h_{t_0,l,t_1}(t)>0$,  $h_{t_0,l,t_1}'(t)<0$, solving the  ordinary differential equation     
 \begin{equation}\label{varphi}
-\frac{h_{t_0,l,t_1}''}{h_{t_0,l,t_1}}=\varphi_{t_0,l,t_1},
\end{equation} with boundary conditions $h_{t_0,l,t_1}(t_0)=\sin(t_0)$,  $h_{t_0,l,t_1}(t_1)=l$, where   $\varphi_{l,t_1}(t)$ is a smooth nonincreasing function interpolating between $\varphi_{t_0,l,t_1}(t_0)=1$ and $\varphi_{t_0,l,t_1}(t_1)=-1$, with \begin{equation}\label{BF4}
\int_{t_0}^t\varphi_{t_0,l,t_1}(s)ds>0, \qquad \forall t_1\in ]t_0,2\pi[, \qquad \forall l\in]0,\sin(t_0)[, \qquad \forall t\in]t_0,t_1].\end{equation} Notice also that such a $\varphi_{t_0,l,t_1}$ always exists and $h_{t_0,l,t_1}(t)>0$.
 %=\left(1-\frac{\delta}{2\pi}\right)I^\flat<1$ where $\delta\to2\pi$ 
 %Notice that $h'(t)<0$ is true. because of the Sturm-Liouville-Riccati comparison theory but we do not use this fact. 
 The volume is \begin{equation}\label{BF5}
\mathrm{vol}_g(M)=-\cos(t_0)+\delta_{t_0,l,t_1}+\frac{2\pi l}{I^\flat}\ge-\cos(t_0)+l(t_1-t_0)+\frac{2\pi l}{I^\flat}>2\pi,\end{equation} 
where
 \begin{equation}\label{BF6}
g:=g_{t_0,t_1, l} \ \ \mathrm{and} \ \
 \delta_{t_0,l,t_1}= 2\pi \ \int_{t_0}^{t_1} h_{t_0,l,t_1}(s)ds .\end{equation}Now choose $t_0$ and $t_1-t_0$ possibly smaller such that \begin{equation}\label{BF7}
\mathrm{vol}_g(M)=-\cos(t_0)+\delta_{t_0,l,t_1}+2\pi l=2\pi.\end{equation} 
 Such a choice is always possible because by the continuous dependence on the parameter of the solutions of second order ordinary differential equations and Lebesgue dominated convergence theorem, $\delta_{t_0,l,t_1}$ is a continuous function of the vector $({t_0,l,t_1})$.

 One can see that $(M,g_{t_0,l,t_1})=(M,g)$ has nonincreasing sectional curvature belonging to $[-1,1]$ by construction. Let $\Omega_t=\{(s,\theta)|0\le s\le t\}$ and use again \cite[Lemma 3.1, B]{morgan4}. We know that the absolute minimizers have to be found among the family of domains $\{\Omega_t\}_{t\in]0,+\infty[}$.  By the general Bol--Fiala inequality \cite[Equation 1.3]{BarbosadoCarmo}, we have that for any domain $\Omega\subset\subset M$ with smooth boundary (or any finite perimeter set, by approximation)
 \begin{equation}\label{Eq:Bol-FialaExemple5.4}
 A_g(\partial\Omega) \ge 4\pi\sqrt{\mathrm{vol}_g(\Omega)(4  \pi-\mathrm{vol}_g(\Omega))}.
 \end{equation}
 Thus the function 
 \begin{eqnarray}\label{BF8}
I_{M,g}^\flat(t) & :=  & I_{M,g}^\flat(\Omega_t)=2\pi f(t)\left[\frac1{2\pi\int_{0}^tf(s)ds}+\frac1{4\pi-2\pi\int_{0}^tf(s)ds}\right]\\ \nonumber
 & = &  A_g(\partial\Omega_t)\left[\frac1{\mathrm{vol}_g(\Omega_t)}+\frac1{4\pi-\mathrm{vol}_g(\Omega_t)}\right]\\ \nonumber
 & \ge & \frac{4\pi}{\sqrt{\mathrm{vol}_g(\Omega_t)(4\pi-\mathrm{vol}_g(\Omega_t))}}\ge 1=I_{S^2,\rho}^\flat= I^\flat, \nonumber
\end{eqnarray} 
 for all $ t\in]0,+\infty[,$ since the function $\tau\mapsto\frac{4\pi}{\sqrt{\tau(4\pi-\tau)}}$ has as minimum value $1$ and attains its unique minimum on the interval $]0,4\pi[$ at $\tau=2\pi$.

It is easy to check $I^\flat_{S^2,\rho}(\Omega_t)=I_{M,g}^\flat(\Omega_t)$ for every $t\in]0,\frac\pi2]$. Moreover $I_{M,g}^\flat$ is achieved because $I^\flat_{S^2,\rho}$ is achieved by some domain $\Omega_t$,  with $t\in ]0,\frac\pi2 ]$, namely $\Omega_{\frac\pi2}$. So the same domain $\Omega_t$ that achieves $I^\flat_{S^2, \rho}$ achieves also $I_{M,g}^\flat=I_{M,g}^\flat(\Omega_{\frac\pi2})$. On the other hand arguing as in  \eqref{Eq:Example5.2}  we may conclude that  $(M,g_{t_0,t_1, l})$ has a compact minimizer $I^\flat(\Omega_{\frac\pi2})=I^\flat$, $\Omega_{\frac\pi2}\subset\subset M$ and does not satisfy $(i)-(ii)$ of Theorem \ref{ste4}. 
\end{ex}

\medskip
\medskip

For the other functionals treated in this paper in dimension $2$ it is easy to adapt the construction of Example \ref{newexample} to obtain analogous examples showing the sharpness of our results.

%Again in dimension $3$ and higher it should be not too hard to construct such a kind of counter examples.%\end{rem}

\section*{Acknowledgements} The authors  thank Pierre Pansu for his valuable comments  that helped to improve the original results  of the present paper. Special thanks go to Luis Eduardo Osorio Acevedo for helping us with the figures that appear in the text and to the anonymous reviewer for relevant comments pointing out a mistake (that now we settled) in a previous version of this manuscript. The first author has been partially sponsored by Fapesp (2018/22938-4), and by CNPq (302717/2017-0), Brazil. The second author has been partially supported by PDJ of CNPq in the years 2013--2014 in Brazil, and by NRF  with grants No.CSRU180417322117, ITAL170904261537 in the years 2018--2020 in South Africa.


\begin{thebibliography}{20}

\bibitem{afp}L. Ambrosio, N. Fusco and D. Pallara, \textit{Functions of bounded variation and free discontinuity problems}, Oxford University Press,  Oxford, 2000.

%\bibitem{ambrosio1}L. Ambrosio, N. Gigli, A. Mondino and T. Rajala, Riemannian Ricci curvature lower bounds in metric measure spaces with $\sigma$--finite measure, AriXiv: 1207.4924, 2012, to appear in \textit{Trans.  Amer. Math. Soc.}.

\bibitem{BarbosadoCarmo}J. L. Barbosa and M. do Carmo, A Proof of a General Isoperimetric Inequality for Surfaces, \textit{Math. Z.} \textbf{162} (1978), 245--261.

\bibitem{BavardPansu}C. Bavard and P. Pansu, Sur le volume minimal de $\mathbb{R}^2$, \textit{Ann. Sci. \'{E}cole Norm. Sup. (4)} \textbf{19} (1986), 479--490.

\bibitem{cao}I. Benjamini and J. Cao, A new isoperimetric comparison theorem for surfaces of variable curvature, \textit{ Duke Math. J.} \textbf{85} (1996), 359--396.

\bibitem{beme} P. B\'erard and D. Meyer, Inegalit\'es isop\'erim\'etrique et applications, \textit{Ann. Sci \'{E}cole Norm. Sup. (4)} \textbf{15}(1982), 531--542.

\bibitem{caccioppoli}R. Caccioppoli, Misura e integrazione sugli insiemi dimensionalmente orientati, Note I e II, \textit{Atti Accad.Naz.Lincei, VIII, Ser. Rend.Cl.Sci.Fis. Mat. Nat.}\textbf{12} (1952), 3--12,137--146.


\bibitem{chavel}	I. Chavel, \textit{Riemannian geometry: a modern introduction}, 2nd edition, Cambridge University Press, Cambridge, 2006.

\bibitem{chow}B. Chow and D. Knopf, \textit{The Ricci flow: an introduction}, Mathematical Surveys and Monographs, Vol. 110, Amer. Math. Soc., Providence, 2004.

%\bibitem{Cooper10}A. Cooper, A compactness theorem for the second fundamental form, Preprint, ArXiv:1006.5697v4, 2011.

\bibitem{delpino1}P. Daskalopoulos and M.A. del Pino, On a singular diffusion equation, \textit{Commun. Anal. Geom.} \textbf{3}, (1995), 523--542.

\bibitem{delpino2}P. Daskalopoulos and M.A. del Pino, Type II collapsing of maximal solutions to the Ricci flow in $\mathbb{R}^2$, \textit{Ann. Inst. H. Poincar\'e Anal. Non Lin\'eaire} \textbf{24} (2007), 851--874.

\bibitem{hamilton1}P. Daskalopoulos and R.S. Hamilton, Geometric estimates for the logarithmic fast diffusion equation, \textit{Comm. Anal. Geom.} \textbf{12} (2004), 143--164.

\bibitem{hamilton2}P. Daskalopoulos, R.S. Hamilton and N. Sesum, Classification of compact ancient solutions to the Ricci flow on surfaces,   \textit{  J. Differential Geom.}\textbf{ 91} (2012), 171--214.

%\bibitem{hsu} Shu-Yu Hsu, Minimizer of an isoperimetric ratio on a metric on $R^2$ with finite total area,  arXiv:1001.4241

\bibitem{degiorgi} E. De Giorgi, Definizione ed espressione analitica del perimetro di un insieme, \textit{Atti Accad. Naz. Lincei, VIII. Ser., Rend., Cl. Sci. Fis. Mat. Nat.} \textbf{14} (1953), 390--393.

%\bibitem{floresNardulliContinuidade} A.E. Mu\~noz Flores and S. Nardulli, Continuity and differentiability properties of the isoperimetric profile in complete noncompact Riemannian manifolds with bounded geometry, preprint, 2014, arXiv:1404.3245.

%\bibitem{floresNardulliCompacidade} A.E. Mu\~noz Flores and S. Nardulli, Generalized compactness for finite perimeter sets and applications to the isoperimetric problem, preprint, 2015, arXiv:1504.05104.







%\bibitem{ambrosio2}L. Ambrosio, A. Mondino and G. Savar\'e, On the Bakry--\'Emery condition, the gradient estimates and the Local--to--Global property of $RCD^*(k,n)$ metric measure spaces, preprint, ArXiv:1309.4664v1, 2013.

%\bibitem{aubin}T. Aubin, \textit{Nonlinear analysis on manifolds. Monge--Amp\'ere equations}, Grundlehren der Mathematischen Wissenschaften, 252, Springer, 1982.

%\bibitem{bakry1}D. Bakry, T. Coulhon, M. Ledoux and L. Saloff--Coste, Sobolev inequalities in disguise, \textit{Indiana Univ. Math. J.} \textbf{44} (1995), 1033--1074.

%\bibitem{bakry2}D. Bakry and M. Ledoux, Sobolev inequalities and Myers's diameter theorem for an abstract Markov generator, \textit{Duke Math. J.} \textbf{85} (1996), 253--270.

%\bibitem{bau1}F. Bauer, P. Horn, Y. Lin, G. Lippner, D. Mangoubi and S.--T. Yau , Li--Yau inequality on graphs,preprint,  arXiv:1306.2561, 2013.

%\bibitem{bau2}F. Bauer, M. Keller and R.K. Wojciechowski, Cheeger inequalities for unbounded graph Laplacians, preprint, ArXiv: 1209.4911, to appear in  \textit{J. Europ. Math. Soc.}

%\bibitem{bau3}F. Bauer, B. Hua and J. Jost,  The dual Cheeger constant and spectra of infinite graphs, arXiv:1207.3410.

%\bibitem{bau3} F. Bauer,  B. Hua and M. Keller, On the $l^p$ spectrum of Laplacians on graphs,\textit{Adv. Math.} \textbf{248} (2013), 717--735.

%\bibitem{bau4}F. Bauer,  Normalized graph Laplacians for directed graphs, \textit{Linear Algebra Appl.} \textbf{436}  (2012),  4193--4222.

%\bibitem{bau5}F. Bauer, J. Jost and S. Liu,  Ollivier--Ricci curvature and the spectrum of the normalized graph Laplace operator, \textit{Math. Res. Lett.} \textbf{19} (2012), 1185--1205.


%\bibitem{chung1}F.R.K. Chung, \textit{Spectral Graph Theory}, CBMS Regional Conference Series in Mathematics 92, AMS publications, 1996.

%\bibitem{chung2}F.R.K. Chung and S.--T. Yau, Eigenvalues of graphs and Sobolev inequalities, \textit{Combin. Probab.  Comp.} \textbf{4} (1995), 11--26.

%\bibitem{chung3}F.R.K. Chung and S.--T. Yau, A Harnack inequality for homogeneous graphs and subgraphs, \textit{Comm. Anal. Geom.} \textbf{2} (1994), 627--640.


%\bibitem{chung4}F.R.K. Chung, A. Grigor'yan and S.--T. Yau, Higher eigenvalues and isoperimetric inequalities on riemannian manifolds and graphs, \textit{Comm. Anal. Geom.} \textbf{8} (2000), 969--1026.

%\bibitem{delmotte}T. Delmotte, Parabolic Harnack inequality and estimates of Markov chains on graphs,
%\textit{Rev. Math. Iberoamericana} \textbf{15} (1999), 181--232.

%\bibitem{diaconis}P. Diaconis and L. Saloff--Coste, Nash inequalities for finite Markov chains,\textit{J. Theor. Probab.} \textbf{9} (1996), 459--510.

%\bibitem{garofalo}N. Garofalo and A. Mondino, Li--Yau and Harnack type inequalities in $RCD^*(k,n)$ metric measure spaces, ArXiv: 1306.0494v1, 2013. 

%\bibitem{bau6}A. Georgakopoulos, S. Haeseler, M. Keller, D. Lenz and R. K. Wojciechowski, Graphs of finite measure, arXiv:1309.3501. 
 

%\bibitem{hebey}E. Hebey, \textit{Nonlinear analysis on manifolds: Sobolev spaces and    inequalities}, Courant Lecture Notes in Mathematics, Vol.5, New York University Courant Institute of Mathematical Sciences,New York, 1999.

%\bibitem{he2}E. Hebey and M. Vaugon, Sobolev spaces in the presence of symmetries, \textit{J. Math. Pures Appl.} \textbf{76} (1997), 859--881.


%\bibitem{hes}S. Haeseler, \textit{Heat kernel estimates and related inequalities on metric graphs},  arXiv:1101.3010, 2011.


%\bibitem{hofrus}K.H. Hofmann and F.G. Russo, The probability that $x$ and $y$ commute in a compact group, \textit{Math. Proc. Cambridge Phil. Soc.} \textbf{153} (2012), 557--571.


%\bibitem{lin1}Y. Lin and S.--T. Yau, Ricci curvature and eigenvalue steimate on locally finite graphs, \textit{Math. Res. Lett.} \textbf{17} (2010), 345--358.

%\bibitem{lin2}Y. Lin, L. Lu and S.--T. Yau, Ricci curvature of graphs, \textit{Tohoku Math. J.} \textbf{63} (2011), 605--627.

%\bibitem{nardulli2}A. Mondino and S. Nardulli, Existence of isoperimetric regions in noncompact riemannian manifolds under Ricci or scalar curvature conditions, ArXiv: 1210.0567v1, 2012.


\bibitem{ritore3}M. Galli and M. Ritor\'e,
Existence of isoperimetric regions in contact sub-Riemannian manifolds,\textit{J. Math. Anal. Appl.} \textbf{397} (2013), 697--714.




%\bibitem{grimaldi}R.Grimaldi, S. Nardulli and P. Pansu, Semianalyticity of isoperimetric profiles,\textit{Differ. Geom. Appl.} \textbf{27} (2009), 393--398.


%\bibitem{grove-petersen}K. Grove and P. Petersen, Manifolds near the boundary of existence, \textit{J.Differential Geom.} \textbf{33} (1991), 379--394.

%\bibitem{gromov1} M. Gromov, Groups of polynomial growth and expanding maps, \textit{Inst. Hautes Etudes Sci. Publ. Math.} \textbf{53} (1981), 53--73.

%\bibitem{gromov2} M. Gromov,  \textit{Metric structures for Riemannian and non-Riemannian spaces}, volume 152,  Progress in Mathematics, Birkh\"auser, Boston, 1999


\bibitem{hamilton3}R.S. Hamilton,  Isoperimetric estimates for the curve shrinking flow in the plane, \textit{Modern methods in complex analysis} (Princeton, NJ, 1992) Ann. of Math. Stud., vol. 137, Princeton Univ. Press, Princeton, NJ, 1995, pp. 201--222.

\bibitem{hamilton4}R.S. Hamilton, An isoperimetric estimate for the Ricci flows on the two--sphere,  \textit{Modern methods in complex analysis} (Princeton, NJ, 1992) Ann. of Math. Stud., vol. 137, Princeton Univ. Press, Princeton, NJ, 1995, pp. 191--200 .

\bibitem{morgan4}H. Howards, M. Hutchings and F. Morgan, The isoperimetric problem on surfaces of revolution of decreasing Gauss curvature, \textit{Trans. Amer. Math. Soc. } \textbf{352}  (2000), 4889--4909.



\bibitem{hsu} S.-Y. Hsu, Minimizer of an isoperimetric ratio on a metric on $\mathbb{R}^2$ with finite total area, \textit{Bull.  Sci. Math. } \textbf{8} (2018), 603--617.


 
%\bibitem{kos}C. Kosniowski, \textit{A first course in algebraic topology}, Cambridge University Press, Cambridge, 1980.

\bibitem{maggi}F. Maggi, \textit{Set of finite perimeter and geometric variational problems. Introduction to geometric measure theory}, Cambridge University Press, Cambridge, UK, 2012.

\bibitem{miranda}M. Miranda, D. Pallara, F. Paronetto and M. Preunkert, Heat semigroup and functions of bounded variation on Riemannian manifolds, \textit{J. Reine Angew. Math.} \textbf{613} (2007), 99--119. 


\bibitem{morgan1}F. Morgan, Regularity of isoperimetric hpersurfaces in Riemannian manifolds, \textit{Trans. Amer. Math. Soc.} \textbf{ 355} (2003),   5041--5052.


\bibitem{morgan2}F. Morgan, \textit{Geometric measure theory}, Academic Press, San Diego,  2000.

\bibitem{morgan3}F. Morgan and M. Ritor\'e, Isoperimetric regions in cones, \textit{Trans. Amer. Math. Soc.} \textbf{354} (2002), 2327--2339.


\bibitem{localhoelder}A.E. M.u\~noz Flores and S. Nardulli, Local H\"older continuity of the isoperimetric profile in complete noncompact Riemannian manifolds with bounded geometry, \textit{Geom. Dedicata} \textbf{201} (2019), 1--12.


\bibitem{nardulli1} S. Nardulli, The isoperimetric profile of a noncompact Riemannian manifold for small volumes, \textit{Calc. Var. Partial Differential Equations} \textbf{49} (2014), 173--195.

\bibitem{nardulli2}S. Nardulli, Regularity of solutions of the isoperimetric problem that are close to a smooth manifold,   \textit{Bull. Braz. Math. Soc. (N.S.)}  \textbf{49}(2018), 199--260.

%\bibitem{NardulliOsorioIMRN} S. Nardulli and L. E. Osorio Acevedo, Sharp Isoperimetric Inequalities for Small Volumes in Complete Noncompact Riemannian Manifolds of Bounded Geometry Involving the Scalar Curvature, \textit{Int. Math. Res. Not. IMRN}  \textbf{15} (2020), 4667--4720.


\bibitem{nardulli-pansu}S. Nardulli and P. Pansu, A discontinuous isoperimetric profile for a complete Riemannian manifold, \textit{Ann. Sc. Norm. Super. Pisa Cl. Sci. (5)} \textbf{XVIII} (2018), 537--549.



\bibitem{pansu}P. Pansu, Regularity of the isoperimetric profile of compact Riemannian surfaces, \textit{Ann. Inst. Fourier (Grenoble)} \textbf{48} (1998), 247--264.


\bibitem{papasoglu}P. Papasoglu and E. Swenson, A surface with discontinuous isoperimetric profile and expander manifolds,
\textit{ Geom. Dedicata} \textbf{206} (2020),43--54.


\bibitem{petersen}P. Petersen, \textit{Riemannian geometry},  2nd ed, Springer, Heidelberg, 2006.

\bibitem{Polulyakh}E. Polulyakh, On the theorem converse to Jordan's curve theorem, \textit{Methods Funct. Anal. Topology}  \textbf{6} (2000),   56--69.

\bibitem{ritore1}M. Ritor\'e and C. Rosales, Existence and characterization of regions minimizing perimeter under a volume constraint inside euclidean cones, \textit{Trans. Amer. Math. Soc.} \textbf{356} (2004), 4601--4622.

\bibitem{ritore2}M. Ritor\'e, Constant geodesic curvature curves and isoperimetric domains in rotationally symmetric surfaces, \textit{Comm. Anal. Geom.} \textbf{9} (2001), 1093--1138.


%\bibitem{nash}J. Nash, Continuity of solutions of parabolic and elliptic equations, \textit{Amer. J. Math.} \textbf{80} (1958), 931--934.

%\bibitem{simon}L. Simon, \textit{An introduction to geometric measure theory},  2014, Stanford University, Available online at:web.stanford.edu/class/math285/ts-gmt.pdf

%\bibitem{sobolev}S. Sobolev, Sur un th\'eor\'eme d'analyse fonctionelle, \textit{Mat. Sbornik} \textbf{46} (1938),  471--496.

\bibitem{topping}P. Topping, The isoperimetric inequality on a surface, \textit{Manuscripta Math.} \textbf{100} (1999), 23--33.


\bibitem{zheng} Y. Zheng, Uniform Lipschitz continuity of the isoperimetric profile of compact surfaces under normalized Ricci flow, preprint, 2020,  arXiv:2001.00341  


\end{thebibliography}
\end{document}